\documentclass{amsart}
\usepackage{amsmath,amssymb,amsthm,mathtools,mathdots,nicefrac,tikz,enumitem,url}

\usepackage[a4paper,top=2.5cm,bottom=2cm,left=2cm,right=2cm,marginparwidth=1.75cm]{geometry}

\usepackage{thmtools}
\usepackage{thm-restate}
\usepackage{hyperref}
\usepackage{cleveref}

\usepackage{algorithm}
\usepackage[noend]{algpseudocode}
\usepackage{float}
\usepackage[caption = false]{subfig}

\theoremstyle{definition} 
\newtheorem{theorem}{Theorem}[section]
\newtheorem{corollary}{Corollary}[section]
\newtheorem{lemma}{Lemma}[section]
\newtheorem{proposition}{Proposition}[section]

\newtheorem{example}{Example}[section]
\newtheorem{remark}{Remark}[section]
\newtheorem{conjecture}{Conjecture}[section]

\newcommand{\sn}{\mathrm{sn}}

\DeclareMathOperator{\val}{val}
\DeclareMathOperator{\tw}{tw}
\DeclareMathOperator{\gon}{gon}

\DeclareMathOperator{\scw}{scw}

\DeclareMathOperator{\lcm}{lcm}

\usepackage{mathrsfs}

\title{The Gonality of Chess Graphs}
\author{Nila Cibu, Kexin Ding, Steven DiSilvio, Sasha Kononova, Chan Lee, Ralph Morrison, and Krish Singal}
\date{}

\begin{document}

\maketitle

\begin{abstract}
   Chess graphs encode the moves that a particular chess piece can make on an $m\times n$ chessboard.  We study through these graphs through the lens of chip-firing games and graph gonality.  We provide upper and lower bounds for the gonality of king's, bishop's, and knight's graphs, as well as for the toroidal versions of these graphs.  We also prove that among all chess graphs, there exists an upper bound on gonality solely in terms of $\min\{m,n\}$, except for queen's, toroidal queen's, rook's, and toroidal bishop's graphs.
\end{abstract}

\section{Introduction}


The game of chess has a long history of inspiring problems in math. How many queens (or bishops, or knights) can be placed on a chessboard (not necessarily $8\times 8$) such that no pair of them can attack each other? Similarly, what is the smallest number of queens (or bishops, or knights) that can be placed so that every square on the board may be attacked in one move? Can a knight visit each square on an $m\times n$ chessboard exactly once, returning to its starting position? These questions can be framed and studied in the language of graph theory, and relate to such topics as independence number, domination number, and Hamiltonian cycles.

Choose a chess piece, either king, queen, rook, bishop, or knight\footnote{We omit pawns from consideration, since the movement of a pawn is not symmetric between two squares, so an undirected graph does not suffice to encode their movements,}.  The $m\times n$ chess graph associated to that piece has $mn$ vertices, corresponding to the squares of an $m\times n$ chessboard. Two vertices are connected by an edge if and only if the chess piece in question can move between the two corresponding squares. For example, in a king's graph, each vertex is adjacent to at most eight other vertices, corresponding to two horizontal, two vertical, and four diagonal moves. Toroidal chess graphs are constructed similarly, with moves allowed that glue the boundaries of the chessboard to form a torus.  Illustrated on the left in Figure \ref{figure:4x5_kings_diagram} is a $4\times 5$ chessboard with a king on one of the squares.  The king's moves from that square under traditional chess rules are illustrated with solid arrows; when toroidal moves are allowed, three more moves are possible, illustrated with dashed arrows.  On the right are the $20$ vertices of the $4\times 5$ king's graph; the edges coming from one vertex are illustrated, with dashed edges for those only present in the toroidal king's graph. 

\begin{figure}[hbt]
    \centering
    \includegraphics[scale=1.5]{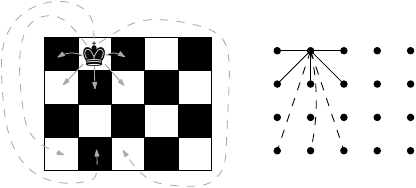}
    \caption{Legal king's moves from one vertex and corresponding edges on $4 \times 5$ king's graph, with dashed edges for the additional edges in the $4\times 5$ toroidal king's graph.}
    \label{figure:4x5_kings_diagram}
\end{figure}

Note that on a toroidal board, a piece may be able to travel between two vertices with multiple routes; for example, on a $2\times n$ chess board, a knight moving two squares spots up and one the right has the same effect as moving two squares down and one to the right. We follow the convention that only one edge connects the two corresponding vertices, although an interesting variant of our problems would be to encode such repeated moves with repeated edges (turning our graph into a multigraph). For some graphs, like rook's graphs, making the board toroidal has no impact at all; for others, such as bishop's graphs, making the board toroidal increases the number of edges and changes the graph structure drastically.

In this paper we study chess graphs through the lens of chip-firing games and graph gonality.  The (divisorial) gonality of a graph, a discrete analog of the gonality of an algebraic curve, is the minimum degree of a positive rank divisor on that graph. Phrased in the language of chip-firing games, the gonality of a graph $G$ is the smallest number of chips that can be placed on $G$ so that a chip can be placed on any vertex through chip-firing moves, without any other vertex being in debt.  This invariant has been of great interest due to its connection to algebraic geometry \cite{baker}, and much work has been done to determine the gonality of different families of graphs.

Several families of chess graphs have already been the subject of study from this perspective. The gonality of rook's graphs was first explored in \cite{aidun2019gonality}, with gonality determined for $m\times n$ graphs with $\min\{m,n\}\leq 5$. This result was generalized in \cite{rooks_gonality}, where gonality was determined for all rook's graphs.  More recently, the gonality of queen's graphs and toroidal queen's graphs was determined in \cite{queens_graph_gonality}.  In this paperwe study the remaining chess graphs, namely the traditional and toroidal versions of king's, bishop's, and knight's graphs arising from $m\times n$ chessboards. The $3\times 4$ non-toroidal versions of these graphs are illustrated in the first row Figure \ref{figure:chess_ensemble}; the second row shows the edges that are gained by the upper left vertex when expanding to the toroidal version. 

\begin{figure}[hbt]
    \centering
\includegraphics{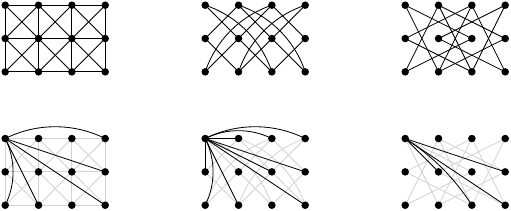}
    \caption{The $3\times 4$ king's, bishop's, and knight's graphs; and the edges gained by the upper left vertex in the toroidal versions of those graphs.}
    \label{figure:chess_ensemble}
\end{figure}



We compute gonality exactly for all toroidal bishop's graphs, as well as for king's and toroidal king's graphs where the number of columns is sufficiently high compared to the number of rows. We also find upper and lower bounds on gonality for king's, toroidal king's, bishop's, knight's, and toroidal knight's graphs.

We summarize our results in the following two theorems; in both we assume $m\leq n$.  The first enumerates our contribution to known bounds on gonality.

\begin{theorem}\label{theorem:summary_of_bounds}
    We have the following bounds on the gonality of chess graphs.
    \begin{itemize}
        \item King's graphs: $\gon(K_{m\times n}) \leq 3m-2$, with equality for $n\geq 3m-2$.
        \item Toroidal king's graphs: $\gon(K^t_{m\times n}) \leq 6m$, with equality for $n\geq 6m$.
        \item Bishop's graphs: $\gon(B_{m\times n}) \leq \frac{m^2(m^2-1)}{6}$; and for $n\gg m$ we have $\gon(B_{m\times n}) \geq \frac{1}{3}(m'-1)m'(m'+1)$, where $m'=2\lfloor m/2\rfloor$.
        \item Toroidal bishop's graphs: $\gon(B_{m\times n}^t) = mn-\gcd(m,n)$.
        \item Knight's graphs: $\gon(N_{m\times n}) \leq 10m-12$ for $m\geq 4$, with smaller bounds for $m \leq 3$; and for $m\geq 4$ and $n\gg m$ we have $\gon(N_{m\times n}) \geq 6m-8$.
        \item Toroidal knight's graphs: $\gon(N_{2\times n}^t) \leq 20$ and $\gon(N_{m\times n}^t)\leq 20m$ for $m\geq 3$; and for $m\geq 5$ and $n\gg m$ we have $\gon(N_{m\times n}^t) \geq 12m$.
    \end{itemize}
\end{theorem}

The second theorem gives a qualitative summary of how different graphs behave as we fix $m$ and vary $n$.

\begin{theorem}\label{theorem:bound_in_just_m} Let $G_{m\times n}$ be an $m\times n$ chess graph from one of the following families: rook's, queen's, toroidal queen's, king's, toroidal king's, bishop's, toroidal bishop's, knight's, and toroidal knight's.  Then, as we vary $n$, there exists a bound on $\gon(G_{m\times n})$ solely in terms of $m$ if and only if the graph is not rook's, queen's, toroidal queen's, or toroidal bishop's.
\end{theorem}

\begin{proof}
    The existence of bounds for king's, toroidal king's, bishop's, knight's, and toroidal knight's follow from Theorem \ref{theorem:summary_of_bounds}.  The non-existence of a bound for toroidal knight's graph follows from our result in Theorem \ref{theorem:summary_of_bounds} that $\gon(B_{m\times n}^t)=mn-\gcd(m,n)$, which is unbounded as $n$ increases while $m$ remains fixed.  The non-existence of bounds for the remaining graphs comes  from previous work: it was shown in \cite{rooks_gonality} that the $m\times n$ rook's graph has gonality $(m-1)n$ for $m\leq n$; and it was shown in \cite{queens_graph_gonality} that the $m\times n$ queen's graph has gonality $mn-m$ for $m\leq n$ and $(m,n)\notin\{(2,2),(3,3)\}$, and that the $m\times n$ toroidal queen's graph has gonality at least $mn-m$ for $m\leq n$.
\end{proof}

Our paper is organized as follows. In Section \ref{section:background_and_definitions} we present background and definitions. In Section \ref{section:kings} we present our results on king's graphs, followed by toroidal king's graphs in Section \ref{section:toroidal_kings}. In Section \ref{section:bishops_graphs} we present our results on bishop's graphs, followed by toroidal bishop's graphs in Section \ref{section:toroidal_bishops_graphs}. In Section \ref{section:knightsgraphs} we present our results on knight's graphs, followed by toroidal knight's graphs in Section \ref{section:toroidal_knights_graphs}. We present directions for future work in Section \ref{section:future_directions}, and several computations referred to in earlier sections in the Appendices.

\medskip

\noindent \textbf{Acknowledgements.} The authors were supported by Williams College and the SMALL REU, as well as by the  NSF via grants DMS-2241623 and DMS-1947438. The second author was supported by Brown University.

\section{Background and definitions}
\label{section:background_and_definitions}

Let $G$ be a graph consisting of a finite set $V=V(G)$ of vertices and a finite multiset $E=E(G)\subseteq \{\{v,w\}\,|\, v,w\in V,v\neq w\}$ of edges.  If $E$ has no repeated edges, we call $G$  \emph{simple} graph.  If a vertex $v$ is in an edge $e$, we say $e$ is \emph{incident} to $v$, and if $\{v,w\}\in E$, we say $v$ and $w$ are \emph{adjacent}. 
We refer to the number of edges incident to $v$ as the \emph{valence}\footnote{This is also referred to as the \emph{degree} of $v$; however, the term degree is used with a different meaning in chip-firing games.}
of $v$, denoted $\val(v)$. For disjoint subsets $A,B\subset E(G)$, we let $E(A,B)$ be the set of edges with a vertex in $A$ and a vertex in $B$; and we let $E(v,w)=E(\{v\},\{w\})$. If for all pairs of vertices $v,w\in V$ we can find vertices $v=v_0,v_1,v_2,\ldots,v_k=w$ such that $v_i$ and $v_{i+1}$ are adjacent for all $i$, we say $G$ is \emph{connected}; otherwise it is \emph{disconnected}.  A subset $V'\subset  V(G)$ of vertices is said to be \emph{connected} if the subgraph $G'$ with vertex set $V'$ and edge set $\{\{v,w\}\in E\,|\,v,w\in V'\}$ is  connected graph.

An $m\times n$ chess graph has vertex set $$V=\{(i,j)\,|\,i,j\in\mathbb{Z}, 1\leq i\leq m,1\leq j\leq n\}.$$
When it is more notationally convenient, we may denote the vertex $(i,j)$ as $v_{i,j}$. 
We depict these vertices in a grid, with $m$ rows and $n$ columns, treating the upper left corner as the vertex $(1,1)$, with $i$ denoting the row of the vertex and $j$ denoting the column.  Identifying these vertices with the squares in an $m\times n$ chess board, two vertices are adjacent if and only if the chess piece in question can move directly between the corresponding squares.  In particular, we have the following adjacency rules for when distinct vertices $(i_1,j_1)$ and $(i_2,j_2)$ are adjacent in our chess  graphs.
\begin{itemize}
  \item King's graph:  when $i_1-i_2,j_1-j_2\in\{-1,0,1\}.$
  \item Bishop's graph: when $(i_1-i_2,j_1-j_2)$ is an integer multiple of either $(1,1)$ or $(1,-1)$.
  \item Knight's graph: when $(i_1-i_2,j_1-j_2)\in\{(\pm 1,\pm 2),(\pm 2,\pm 1)\}$.
  \item Rook's graph: when $i_1=i_2$ or $j_1=j_2$.
  \item Queen's graph: when either $(i_1-i_2,j_1-j_2)$ is an integer multiple of either $(1,1)$ or $(1,-1)$, or $i_1=i_2$ or $j_1=j_2$.
\end{itemize}
An $m\times n$ toroidal chess graph has the same vertex set and the same adjacency rules, except that all equalities are considered modulo $m$ for the first coordinate and modulo $n$ for the second coordinate.  When representing these graphs in figures, we either draw the graph with vertices and edges, or when such a figure would be hard to digest visually simply present a chessboard, with a caption making clear which chess piece we are considering.

We now present the theory of divisors on graphs, and refer the reader to \cite{sandpiles} for more details. We remark that most treatments of chip-firing games assume our graph is connected; we do not assume that here, and will note when our difference in assumptions has an impact on various results.

Given a graph $G$, a divisor is a map \(D:V(G) \rightarrow \mathbb{Z}\), i.e. an assignment of an integer to each vertex of $G$. A divisor is thought of intuitively as a collection of poker chips placed on each vertex, where $v$ receives $D(v)$ chips, with negative integers representing debt. If $D(v)\geq 0$ for all $v$, we say that $D$ is \emph{effective}. The total number of chips on the graph, $\sum_{v\in V(G)} D(v)$, is referred to as the \textit{degree} of the divisor, denoted $\deg(D)$.  The set of all divisors forms a group under the operation defined by vertex-wise addition (i.e. $(D+D')(v)=D(v)+D'(v)$; this group is isomorphic to the free abelian group on $V(G)$.

Given a divisor $D$ on a graph $G$ and a vertex $v\in V(G)$, we can transform $D$ into another divisor $D'$ through a \emph{chip-firing move at $v$}.  This removes $\val(v)$ chips from $v$ and transfers them to those vertices adjacent to $v$, one along each edge incident to $v$. More formally, $D'(v)=v-\val(v)$, and $D'(w)=D(w)+|E(v,w)|$ for all $w\neq v$.  Two divisors $D$ and $D'$ are said to be \emph{equivalent}, written $D\sim D'$, if $D'$ can be obtained from $D$ via some sequence of chip-firing moves.  This relation is an equivalence relation on the set of all divisors.

A collection of vertices can be fired in any order without changing the final outcome.  This leads us to define \emph{set-firing moves}:  given $A\subset V(G)$, set-firing $A$ means to fire the vertices of $A$ in any order.  This results in a net flow of chips only along the edges leaving the set, whereas for $v,w\in A$ the movement of chips from $v$ to $w$ cancels with the movement of chips from $w$ to $v$.
Examples of two set-firing moves are depicted in Figure \ref{fig:set-firing}. Since the divisors differ by a sequence of chip-firing moves, all three are equivalent.

\begin{figure}[hbt]
    \centering
    \includegraphics[width = 0.8\textwidth]{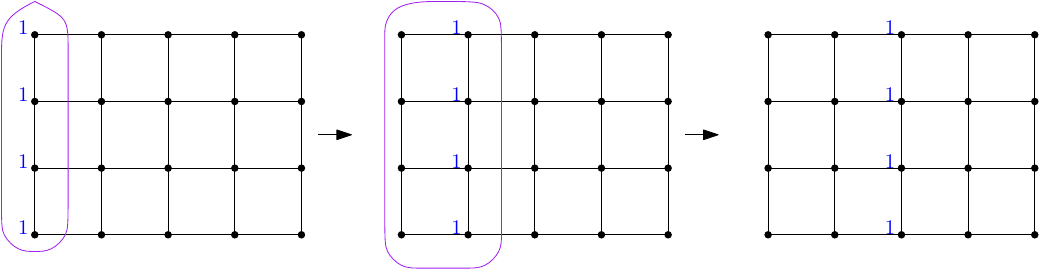}
    \caption{Two examples of set-firing moves, transforming a divisor $D$ into a divisor $D'$ and then into $D''$.  The first set-firing move fires all vertices in the first column; the second set-firing move fires all vertices in the first two columns.}
    \label{fig:set-firing}
\end{figure}

We now define the \emph{rank} of a divisor $D$, denoted $r(D)$.  If $D$ is not equivalent to any effective divisor, we say $r(D)=-1$.  Otherwise, $r(D)$ is the maximum integer $k$ such that, for every effective divisor $D'$ with $\deg(D')=k$, the divisor $D-D'$ is equivalent to some effective divisor.  Intuitively, we can think of the rank of a divisor as the maximum amount of added debt that divisor can eliminate from the graph, regardless of how that debt is placed.

The \emph{(divisorial) gonality} of a graph $G$, denoted $\gon(G)$, is the minimum degree of a positive rank divisor on a graph, i.e. the minimum degree of a divisor that can eliminate $-1$ debt from any vertex without introducing debt elsewhere.  We can equivalently define $\gon(G)$ as the minimum degree of a divisor $D$ such that for any vertex $v\in V(G)$, there exists $D'\sim D$ with $D'$ effective and $D'(v)\geq 1$.

For example, the divisor illustrated in \ref{fig:set-firing} has positive rank: firing the first column, then the first two, then the first three, and so on can move a chip to any vertex without introducing debt elsewhere.  Since that divisor has degree $4$, it follows that the gonality of that graph is at most $4$.  

\begin{remark}
    If a graph $G$ is not connected, say with connected components $G_1,\ldots,G_k$, then
    \[\gon(G)=\gon(G_1)+\cdots+\gon(G_k).\]
    To see this, note that we may replicate positive rank divisors on each component to obtain a positive rank divisor on the whole graph; and if a divisor $D$ places fewer than $\gon(G_i)$ chips on any component $G_i$, there must be some $v\in V(G_i)$ such that debt cannot be eliminated in $D-v$, since chip-firing may be treated independently on each component.
\end{remark}

While an explicit positive rank divisor provides an upper bound on gonality, finding a lower bound is more difficult.  \textit{A priori}, to show that the gonality of a graph is at least $k$, one must consider all degree $k-1$ divisors, and show that for each there exists at least one vertex $v$ such that that divisor cannot eliminate debt on $v$.  While extremely brute force, this strategy can be implemented for relatively small graphs by leveraging Dhar's burning algorithm \cite{sandpiles}, which refers to the first three steps of the following process.

\begin{enumerate}
    \item Given a graph $G$ and an effective divisor $D$, choose a vertex $q\in G$. Start a ``fire" at this vertex, burning it.
    \item Whenever an edge is incident to a burned vertex, that edge burns.
    \item Whenever a vertex $v$ is incident to more than $D(v)$ burning edges, the vertex $v$ burns.
    \item If the whole graph burns, terminate.  If some set $S$ of vertices is unburned, set-fire $S$, and return to step~(1).
\end{enumerate}

An example of the start of this algorithm is illustrated Figure \ref{fig:dhars}. The fire starts at the vertex $q$, and spreads along incident edges.  The adjacent vertex with $1$ chip does not immediately burn, though the one with $0$ chips does.  Its edges burn, burning the left vertex with one chip, burning its edges.  Here the fire stabilizes, since the two rightmost vertices have at least as many chips as burning edges.  Those two vertices are fired, yielding the divisor on the right.  If we ran the burning process again, the whole  graph would burn (first the left vertex with $1$ chip, then the vertex with $2$ chips, and from there the rest of the graph).  Since the whole graph would burn with no chips on $q$, we know that the illustrated divisor does not have positive rank.
\begin{figure}[hbt]
    \centering
    \includegraphics[scale=0.8]{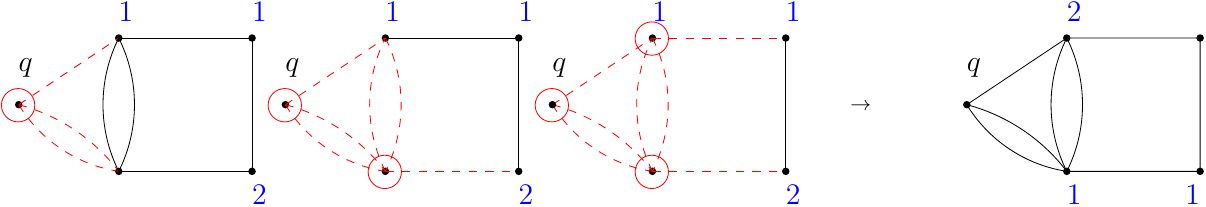}
    \caption{An example of Dhar's Burning algorithm, starting at the vertex $q$.}
    \label{fig:dhars}
\end{figure}

One can show  in general that firing $S$ never introduces debt, and that this process will eventually terminate; see \cite[Chapter 3]{sandpiles} for details.  Moreover, if we start with $q$ such that $D(q)=0$, then the final divisor places a chip on $q$ if and only if there exists an effective divisor $D'\sim D$ such that $D'(q)\geq 1$.  Thus this method can be used to check if an effective divisor $D$ has positive rank:  for each $v\in V(G)$, run the algorithm from $v$, and check if $v$ ends up with at least one chip.  If yes for all $v$, then $r(D)>0$; otherwise $r(D)=0$.  From there, showing that $\gon(G)\geq k$ is accomplished by performing this check for all effective divisors of degree $k-1$.

\begin{remark}  If $G$ is a connected graph, the output of the above algorithm has the property that no nonempty subset of $V(G)-q$ can be set-fired without introducing debt somewhere.  This property does not necessarily hold if $G$ is not connected.
\end{remark}

To avoid such a brute-force approach, one can study graph parameters that serve as lower bounds on gonality.  For instance, the well-studied parameter of treewidth, $\tw(G)$, is known to be a lower bound $\gon(G)$ \cite{debruyn2014treewidth}.  More recently, the scramble number of a graph, $\sn(G)$, was introduced in \cite{new_lower_bound} and shown to be a lower bound on $\gon(G)$.  Since $\tw(G)\leq \sn(G)\leq \gon(G)$ for any graph $G$ \cite[Theorem 1.1]{new_lower_bound}, scramble number is a strictly better lower bound, so we focus on it here.

A \textit{scramble} \(\mathcal{S}\) on a graph $G$ is a nonempty collection of connected subsets (referred to as \emph{eggs}) of $V(G)$. To any scramble $\mathcal{S}$, we associate two numbers:  the hitting number $h(\mathcal{S})$, and the egg-cut number $e(\mathcal{S})$.  A \emph{hitting set} for $\mathcal{S}$ is a set $V'\subset V(G)$ such that every egg in $\mathcal{S}$ has a vertex in $V'$; $h(\mathcal{S})$ is then the minimum size of any hitting set.  An \emph{egg-cut} for $\mathcal{S}$ is a subset $E'\subset E(G)$ such that deleting $E'$ from $G$ disconnects the graph into multiple components, at least two of which contain an egg; $e(\mathcal{S})$ is then the minimum size of any egg-cut.  The \emph{order} of a scramble is defined to be the minimum of the hitting number and the egg-cut number:
\[||\mathcal{S}||=\min\{h(\mathcal{S}),e(\mathcal{S})\}.\]
Finally, the scramble number of a graph is the maximum order of any scramble on the graph.

For an example of a scramble, consider the $4\times 5$ grid graph in Figure \ref{fig:set-firing}.  Let $\mathcal{S}$ consist of five eggs, the first consisting of the four vertices in the first column, the second consisting of the four vertices in the second column, and so on.  Since these eggs do not overlap, we immediately have the the hitting number equals the number of eggs, so $h(\mathcal{S})=5$.  The four edges connecting the first and second columns is an egg-cut, so $e(\mathcal{S})\leq 4$.  To see that $e(\mathcal{S})\geq 4$, consider an arbitrary pair of eggs in $\mathcal{S}$.  We can find four pairwise edge-disjoint paths between them: start at any vertex of the left egg and moving across the horizontal edges until reaching the right egg.  Any egg-cut $E'$ separating those two eggs must include at least one edge from each of these four paths; otherwise the eggs would be in the same component of $G-E'$.  As our choice of eggs was arbitrary, we have that any egg-cut has at least four edges in it, so $e(\mathcal{S})\geq 4$, and thus $e(\mathcal{S})=4$.  This means that $||\mathcal{S}||=\min\{4,5\}=4$, so $\sn(G)\geq 4$.  In fact, since $\gon(G)\leq 4$, we have $4\leq \sn(G)\leq \gon(G)\leq 4$, so $\sn(G)=\gon(G)=4$.

For some graphs, scramble number is strictly smaller than gonality.  In such instances, it is useful to have another upper bound on $\sn(G)$, so that we may rule it out as a tool for reaching a conjectured gonality.  To that end we now present the \emph{screewidth} of a graph, $\scw(G)$, which satisfies $\sn(G)\leq \scw(G)$ \cite{screewidth-og}.

A \emph{tree-cut decomposition} $\mathcal{T}=(T,\mathcal{X})$ of a graph $G$ is a tree $T$ and a function $\mathcal{X}:V(G)\rightarrow V(T)$, i.e. an assignment (not necessarily one-to-one, not necessarily onto) of the vertices of $G$ to the nodes of some tree $T$.  (For clarity, we refer to the vertices and edges of $T$ as its \emph{nodes} and \emph{links}, respectively, reserving \emph{vertices} and \emph{edges} for $G$.)  It is natural to represent a tree-cut decomposition as a thickened drawing of the tree $T$ with the vertices of $G$ inside its nodes, with the edges of $G$ drawn along the unique paths in $T$ between the nodes corresponding to its endpoints.  Two examples of tree-cut decompositions of the $4\times 5$ grid graph are illustrated in Figure \ref{figure:tcd_grid}.

\begin{figure}[hbt]
    \centering
    \includegraphics{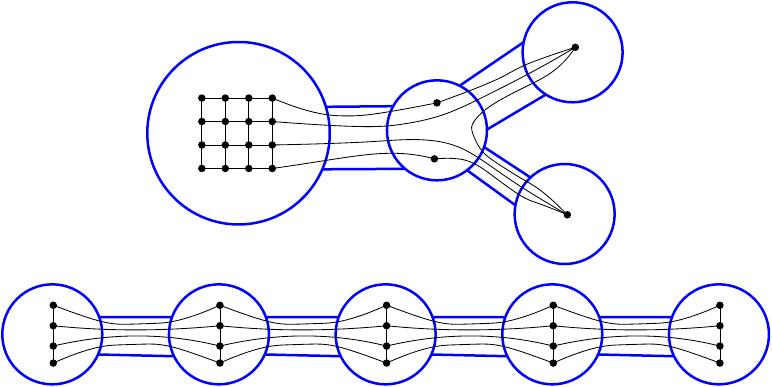}
    \caption{Two tree-cut decompositions of the $4\times 5$ grid.  The first has width $16$; the second has width $4$.}
    \label{figure:tcd_grid}
\end{figure}

Given such a representation of a tree-cut decomposition, we compute several numbers, one for each link and one for each node.  For each link $l$ of $T$, we count the number of edges in $G$ passing through $l$.  For each node $b$ in $T$, we count the number of vertices of $G$ drawn in $b$, and add the number of ``tunneling edges'' of $b$, which have neither endpoint in $b$ but still pass through it.  The maximum of all these numbers is called the \emph{width} of $\mathcal{T}$, denoted $w(\mathcal{T})$.  In our example, the width of the first is computed as the maximum of $4$, $3$, $3$ (from the links), $12$, $5$, $1$, and $1$ (from the nodes), so its width is $12$; the width of the second decomposition is $4$.

The \emph{screewidth} of $G$, written $\scw(G)$, is the minimum width of any tree-cut decomposition of $G$.  For the $4\times 5$ grid graph, we thus have $\scw(G)\leq 4$; since $4=\sn(G)\leq \scw(G)$, we can also deduce that $\scw(G)=4$.  Indeed, without knowing anything about chip-firing games, this could have been used to deduce that $\sn(G)=4$ given that $\sn(G)\geq 4$.

We close this section by recalling several useful bounds on gonality from previous work.  
\begin{lemma}[\S4 in \cite{debruyn2014treewidth}, Theorem 1.1 in \cite{new_lower_bound}] \label{lemma:common_graph_gonalities}
We have the following formulas for gonality.
\begin{itemize}
    \item[(i)] The gonality of a tree is $1$; moreover, trees are the only graphs of gonality $1$.
    \item[(ii)] The gonality of a cycle graph is $2$.
    \item[(iii)] The gonality of a complete graph $K_n$ on $n$ vertices is $n-1$.
    \item[(iv)] The gonality of a complete bipartite graph $K_{m,n}$ is $\min\{m,n\}$.
\end{itemize}
Moreover, in each of the above cases, scramble number equals gonality.
\end{lemma}

For the next two results, we let $\alpha(G)$ denote the independence number of a graph, which is the largest possible size of an independent set of vertices (i.e., a set of vertices with no two adjacent).

\begin{lemma}[Proposition 3.1 in \cite{gonality_of_random_graphs}]\label{lemma:n-alpha(G)} If $G$ is a connected, simple graph on more than one vertex, then $\gon(G)\leq |V(G)|-\alpha(G)$.  The same result holds for $G$ disconnected as long as $G$ has no isolated vertices.
\end{lemma}
The original proof was for $G$ connected, but can be quickly adapted for arbitrary graphs: choose an independent set $S$ with $|S|=\alpha(G)$, and place a single chip on every vertex besides those in $S$.  This divisor has positive rank: to move chips to any vertex $v\in S$, set-fire $\{v\}^C$; since $v$ has at least one neighbor in $G$, and all neighbors of $v$ have a chip and exactly one edge connecting them to $v$, this moves at least one chip to $v$ without introducing debt elsewhere.

\begin{lemma}[Corollary 3.2 in \cite{echavarria2021scramble}] \label{lemma:high_valence}
    Let $G$ be a simple connected graph on $n$ vertices, such that the minimum valence of a vertex of $G$ is at least $\lfloor n/2\rfloor +1$.  Then $\sn(G)=\gon(G)=n-\alpha(G)$.
\end{lemma}

\section{King's Graphs}
\label{section:kings}
A king's graph, denoted $K_{m\times n}$, consists of a grid graph with added diagonals. A $5\times 8$ king's graph is depicted in Figure \ref{fig:23332-king-win}, with a divisor that we will see has positive rank.

\begin{figure}[hbt]
    \centering
    \includegraphics{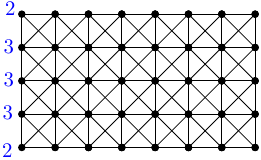}
    \caption{A positive rank divisor on a $5\times 8$ king's graph.}
    \label{fig:23332-king-win}
\end{figure}

\begin{theorem}\label{theorem:gonality_kings}
We have $\gon(K_{m\times n})\leq 3m-2$, with equality for all $n\geq 3m-2$.    
\end{theorem}

\begin{proof}
    For an upper bound, consider the divisor $D$ that places $3$ chips on each vertex of the first column, except that the top and bottom vertices receive only $2$ chips; this is illustrated in Figure \ref{fig:23332-king-win} for $m=5$. Firing all vertices in the first column translates each chip one vertex to the right.  Then, firing the first two columns translates each chip one vertex to the right again.  Continuing in this fashion translates chips across the whole graph, so $D$ has positive rank.  It follows that $\gon(G)\leq \deg(D)=3m-2$.

    If $n\geq 3m-2$, consider the scramble $\mathcal S$ on $K_{m\times n}$ whose eggs are the columns of the graph. Since the $n$ eggs are disjoint, $h(\mathcal S) = n$. Between any two distinct eggs, we can find $3m-2$ pairwise edge-disjoint paths:  $m$ travelling horizontally (starting from any vertex in the left egg), and $2m-2$ alternating between $(1,1)$ and $(1,-1)$ moves (starting from any vertex in the left egg, except that the top vertex cannot start $(1,1)$ and the bottom vertex cannot start $(1,-1)$).  For an egg-cut to separate these eggs, it must contain at least one edge from each of these paths, so $e(\mathcal{S})\geq 3m-2$; and in fact $e(\mathcal{S})= 3m-2$, for instance achieved by deleting all edges between the first two columns. The order of $\mathcal S$ is $\min\{n,3m-2\}=3m-2$, so $\gon(G)\geq \sn(G)\geq ||\mathcal{S}||=3m-2$.  Combined with our upper bound, we have $\gon(G)=3m-2$.
\end{proof}

It is natural to ask whether we can push these results further, to find $\gon(K_{m\times n})$ when $3m-2>n$.  We present several results to this effect for small values of $m$.

\begin{proposition}  We have the following formulas for the scramble number and gonality of $2\times n$ and $3\times n$ king's graphs:

\begin{itemize}
    \item $\sn(K_{2\times n})=\gon(K_{2\times n})=\begin{cases}
        3&\textrm{ if }n=2
        \\ 4&\textrm{ if }n\geq 3.
    \end{cases}$
    \item $\sn(K_{3\times4})=6$ and $\gon(K_{3\times 4})=7$, and  $\sn(K_{3\times n})=\gon(K_{3\times n})=\begin{cases}
        5&\textrm{ if }n=3
        \\ 7&\textrm{ if }n\geq 5.
    \end{cases}$ 
    
\end{itemize}
\end{proposition}
This provides us with $K_{3\times4}$ as the first instance of a king's graph whose gonality cannot be computed using scramble number. 

\begin{proof}
    Since $K_{2\times 2}$ is the complete graph on $4$ vertices, it has gonality and scramble number $3$ by Lemma \ref{lemma:common_graph_gonalities}(iii).  We know $\gon(K_{2\times n})\leq 4$ by Theorem \ref{theorem:gonality_kings}.  For a lower bound, we first argue that $\sn(K_{2\times 3})\geq 4$. A scramble of order $4$ on this graph is given by $\mathcal{S}=\{\{v_{1,1},v_{1,2}\},\{v_{2,1}\},\{v_{2,2}\},\{v_{3,1},v_{3,2}\}\}$, illustrated on the left in Figure \ref{figure:kings_2by_3by}; hitting number is immediate to calculate as $4$, and egg-cut number can be seen to be at least $4$ by finding four pairwise edge-disjoint paths between each pair of eggs.  Since scramble number is subgraph monotone \cite[Proposition 4.5]{new_lower_bound}, we have $4\leq \sn(K_{2\times3})\leq \sn(K_{2\times n})$ for all $n\geq 3$.  This gives a lower bound on gonality, and so $\sn(K_{2\times n})=\gon(K_{2\times n})=4$ for $n\geq 3$

\begin{figure}[hbt]
    \centering
    \includegraphics[scale=1.7]{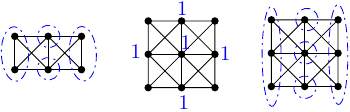}
    \caption{From left to right: a scramble of order $4$ on $K_{2\times3}$; a degree $5$ divisor of positive rank on $K_{3\times3}$; and a scramble of order $5$ on $K_{3\times3}$.}
    \label{figure:kings_2by_3by}
\end{figure}

    To show $\sn(K_{3\times 3})=\gon(K_{3\times 3})=5$, we find an upper bound from the positive rank divisor $D$ illustrated in the middle of Figure \ref{figure:kings_2by_3by}, and a lower bound from the scramble of order $5$ illustrated on the right in the same figure (the analysis is similar to that for $K_{2\times 3}$). For larger $n$, we know $\gon(K_{3\times n})\leq 7$ by Theorem \ref{theorem:gonality_kings}.  For a lower bound when $n\geq 5$, we consider the scramble on $K_{3\times 5}$ illustrated in Figure \ref{figure:k_35_scramble}.  Certainly the hitting number is $7$; to show the egg-cut number is at least $7$, we can find $7$ pairwise edge-disjoint paths between every pair of eggs; see Lemma \ref{lemma:appendix_kings_scramble} in Appendix \ref{appendix:kings} for more details.  This gives $7\leq \sn(K_{3\times 5})\leq \sn(K_{3\times n})$ for all $n\geq 5$, implying $\sn(K_{3\times n})=\gon(K_{3\times n})=7$ for such $n$.

    \begin{figure}[hbt]
        \centering
        \includegraphics[scale=1.7]{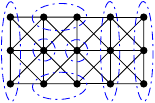}
        \caption{A scramble on $K_{3\times 5}$ with order $7$.}
        \label{figure:k_35_scramble}
    \end{figure}

    Finally, we consider $K_{3\times 4}$.  Figure \ref{figure:scw_for_k34} illustrates a scramble on the left, and a tree-cut decomposition on the right.  The scramble has hitting number $6$ and egg-cut number at least $7$ (the argument is nearly identical to that from the previous paragraph for $K_{3\times 5}$), so the scramble has order $6$.  The tree-cut decomposition has width $6$, giving us $6\leq \sn(K_{3\times 4})\leq \scw(K_{3\times 4})\leq 6$, and thus $\sn(K_{3\times 4})=6$.  To show that $\gon(K_{3\times 4})\geq 7$, we suppose for the sake of contradiction that $K_{3,4}$ has a positive rank divisor of degree $6$.  After carefully choosing an equivalent divisor without too many chips on any given vertex, we choose a particular vertex $v$ with no chips and argue that Dhar's burning algorithm will burn the whole graph, a contradiction to the divisor having positive rank.  This argument is long and technical, so we reserve it for Appendix \ref{appendix:kings} (in particular, see Lemma \ref{lemma:appendix_k34}).

    \begin{figure}[hbt]
        \centering
        \includegraphics[scale=1.7]{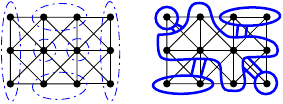}
        \caption{A scramble and a tree-cut decomposition on $K_{3\times 4}$, of order $6$ and width $6$, respectively.}
        \label{figure:scw_for_k34}
    \end{figure}
    
 \end{proof}

We close this section with a discussion of larger knight's graphs, some of which have unknown gonality.  Using a brute-force computation (for instance, using the Chip Firing Interface website \cite{chip_firing_interface}), we can determine that $\gon(K_{4\times 4})=10$ and $\gon(K_{4\times 5})=10$.  However, it is not possible to prove either result using scramble number.  Consider the tree-cut decompositions of $K_{4\times 5}$ in Figure \ref{figure:4x5_kings_scw}. This decomposition has width $8$; since scramble number is non-increasing under taking subgraphs \cite[Proposition 4.5]{new_lower_bound}, we have $\sn(K_{4\times 4})\leq \sn(K_{4\times 5})\leq\scw(K_{4\times 5})\leq8$.  Thus for $4\times n$ king's graphs, we cannot use scramble number to compute all gonalities; however, we find computationally that the eventual gonality of $10$ kicks in immediately for the $4\times 4$ graph, in contrast to having lower gonalities for $2\times 2$, $3\times 3$, and $3\times 4$ graphs.  In light of this computational result, we pose the following conjecture.

\begin{figure}[hbt]
    \centering
    \includegraphics[scale=1.3]{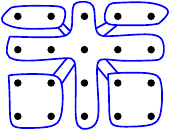}
    \caption{A tree-cut decomposition $K_{4\times 5}$  of width $8$ (edges of the graph omitted for clarity.)}
    \label{figure:4x5_kings_scw}
\end{figure}

\begin{conjecture}\label{conjecture:general_kings}
    For $4\leq m\leq n$, we have $\gon(K_{m\times n})=3m-2$.
\end{conjecture}

\begin{example}  To see that scramble number cannot be used to prove this result for any choice of $m=n\geq 4$, consider the family of tree-cut decompositions on $K_{m\times m}$ illustrated in Figure \ref{figure:square_kings_tree_cut} for $4\leq m\leq 7$. For $m$ odd these decompositions consist of a central plus-shaped bag, with $L$-shaped bags of decreasing size emanating towards the four corners, with two bags connected precisely when two of their vertices share an edge in $K_{m\times m}$.  There are no tunneling edges; the largest bag has $2m-1$ vertices; and the largest link adhesion occurs between the central bag and one of its neighboring bags.  A neighboring bag has $m-2$ vertices, of which $2$ share $2$ edges with the central bag, $1$ shares $5$, and the remainder share $3$.  Thus the maximum link adhesion is $4+5+3(m-4)=3m-6$, which since $m\geq 5$ is larger than $2m-1$, so the width of the tree-cut decomposition is $3m-6$ for odd $m$.  The even case is handled similarly, except the central bag is off-centered and includes a half-column to the upper right, and some of the bags have a more oblong shape.  A similar analysis shows that the largest bag size is $2m-1+\frac{m}{2}=\frac{5m}{2}-1$ and the largest edge adhesion is $4+5+3(m-5)=3m-6$.  The width is therefore $\max\{\frac{5m}{2}-1,3m-6\}$.  In both the even and odd cases, the width is strictly smaller than $3m-2$.  This implies that $\sn(K_{m\times m})\leq \scw(K_{m \times m})<3m-2$ for all $m\geq 4$, so if Conjecture \ref{conjecture:general_kings} holds, it cannot be proved for any square king's graph using scramble number.

\begin{figure}[hbt]
    \centering
    \includegraphics{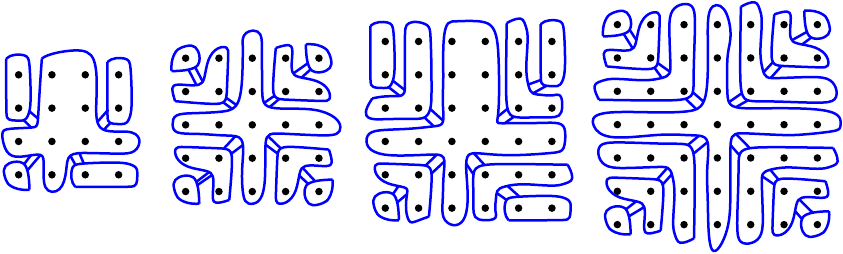}
    \caption{Tree-cut decompositions of $K_{m\times m}$ for $4\leq m\leq 7$, with widths $9$,  $12$, $14$, and $15$, respectively (edges of the king's graphs omitted for clarity).}
    \label{figure:square_kings_tree_cut}
\end{figure}
\end{example}

\section{Toroidal king's graphs}
\label{section:toroidal_kings}
Toroidal king's graphs add edges to the boundary rows and columns of a usual king's graph.  For $3\leq m\leq n$, this yields an $8$-regular graph, with every vertex having $8$ neighbors.  When $m=2$, fewer edges are added, since many of the edges that would be added are already present.  We thus split our results across two theorems.  First we present a lemma on the independence number of  toroidal king's graphs.

\begin{lemma}\label{lemma:toroidal_kings_independence}  The $m\times n$ toroidal king's graph with $m\leq n$ has independence number \[\alpha(K^t_{m\times n})=\begin{cases}
\left\lfloor \frac{m}{2}\right\rfloor \left\lfloor \frac{n}{2}\right\rfloor\textrm{ if $m$ or $n$ is even}
\\ \left \lfloor \frac{n}{2}\left\lfloor \frac{m}{2}\right\rfloor\right\rfloor\textrm{ if $m$ and $n$ are both odd}.
\end{cases}\]
\end{lemma}

\begin{proof}

We first consider $K^t_{2\times n}$.  Every vertex in this graph borders both the other vertex in its column and the four vertices in neighboring columns.  Thus, each column has at most one vertex in an independent set, and no two neighboring columns can both have a vertex in an independent set.  Thus the largest possible size of an independent set is $\lfloor n/2\rfloor$.  Such an independent set exists: simply choose a vertex in every other column, rounded down if $n$ is odd.  Thus $\alpha(K^t_{2\times n})=\lfloor n/2\rfloor$, matching or formula.

For $3\leq m\leq n$, we note that $K^t_{m\times n}$ is the strong product\footnote{The strong product of two graphs $G$ and $H$ has vertex set $V(G)\times V(H)$, with two distinct vertices $(u,v)$ and $(u',v')$ adjacent precisely when either $u=u'$ or $u$ is adjacent to $u'$ in $G$, and either $v=v'$ or $v$ is adjacent to $v'$ in $H$.} of the cycle graphs $C_m$ and $C_n$.  The formulas for this strong product are computed in \cite[\S 4]{independence_product_graphs}, and match our claim.
\end{proof}

\begin{theorem}\label{theorem:2xn_toroidal_kings}
The gonality of a $2\times n$ toroidal king's graph is given by
\[\gon(K^t_{2\times n})=\begin{cases} 2n-1\textrm{ if }n\in\{2,3\}
\\ 6\textrm{ if }n= 4
\\ 8\textrm{ if }n\geq 5.
\end{cases}\]
\end{theorem}

\begin{proof} For $n\in\{2,3\}$, we have that $K^t_{2\times n}$ is the complete graph on $2n$ vertices, and so has gonality $2n-1$ by Lemma \ref{lemma:common_graph_gonalities}(iii).  We note that every vertex in $K^t_{2\times 4}$ has valence $5$, which is strictly more than half the number of vertices.  This lets us apply Lemma \ref{lemma:high_valence} to conclude that $\gon(K^t_{2\times4})=|V(K^t_{2\times4})|-\alpha(K^t_{2\times 4})=8-2=6$.

For $n\geq 5$, we find an upper bound of $8$ on gonality by exhibiting a positive rank divisor of degree $8$.  Choose any column in $K^t_{2\times n}$, and place $4$ chips on each vertex.  Firing both vertices in this column moves $2$ chips to each of the vertices in the adjacent columns.  Then firing the original column together with the two adjacent columns moves $2$ chips to each vertex in the next columns out.  Repeating this process by firing larger and larger sets of columns eventually moves chips throughout the whole graph, so this divisor has positive rank, and $\gon(K^t_{2\times n})\leq 8$.

For our lower bound, we first consider the $2$-uniform scramble $\mathcal{E}_2$ on $K^t_{2\times 5}$, whose eggs are all pairs of adjacent vertices.  By \cite[Lemma 3.2]{uniform_scrambles}, we have $h(K^t_{2\times 5})=|V(K^t_{2\times 5})|-\alpha(K^t_{2\times 5})=10-2=8$.  Let $T\subset E(K^t_{2\times 5})$ be an egg-cut, so that $K^t_{2\times 5}-T$ has two components $G_1$ and $G_2$, say with $n_1$ and $n_2$ vertices, where without loss of generality $2\leq n_1\leq n_2$.  If $n_1=2$, then since each vertex has valence $5$ we know there are at least $2\cdot 5-2\cdot 1=8$ edges in $T$, where the $-2$ comes from the double-counted edge interior to $G_1$.  If $n_1=3$, then there are at most $3$ edges interior to $G_1$, giving at least $3\cdot 5-2\cdot 3=9$ edges in $T$; and similarly if $n_1=4$ we find at least $4\cdot 5-2\cdot 6=8$ edges in $T$. We observe that in any connected five vertex subgraph of $K^t_{2\times 5}$, we can find at least two pairs of vertices (possibly overlapping) that are not in cyclically adjacent columns, meaning that such a subgraph has at most $8$ edges.  Thus if $n_1=5$,  there are at most $8$ edges interior to $G_1$, yielding at least $5\cdot 5 - 2\cdot 8=9$ edges in $T$.  In all cases, $|T|\geq 8$.  As this was an arbitrary egg-cut, we have that $e(\mathcal{E}_2)\geq 8$, so $\sn(K^t_{2\times 5})\geq ||\mathcal{E}_2||=8$.  As scramble number is subgraph monotone, $\sn(K^t_{2\times n})\geq 8$ for all $n\geq 5$.  This provides our desired lower bound on $\gon(K^t_{2\times n})$.
\end{proof}

We now present our most general result for toroidal king's graphs.

\begin{theorem}
For $m\geq 3$, we have $\gon(K_{m\times n}^t) \leq 6m$, with equality for $n\geq 6m$.    
\end{theorem}
\begin{proof}
As with king's graphs in the previous section, we show the upper bound through positive rank divisor and the lower bound through scramble number. 

Choose a column in the graph and place $6$ chips on each vertex, for a total of $6m$ chips; this divisor is illustrated on the left in Figure \ref{fig:tor-king}.  Firing all vertices in this column moves $3$ chips to each of the vertices in the adjacent columns, as illustrated on the right.  Then firing the original column together with the two adjacent columns moves $3$ chips to each vertex in the next columns out.  Repeating this process by firing larger and larger sets of columns eventually moves chips throughout the whole graph, so this divisor has positive rank, and $\gon(K^t_{m\times n})\leq 6m$.


Now assume $n\geq 6m$, and consider the scramble $\mathcal{S}$ on $K^t_{m\times n}$ whose eggs are precisely the columns.  Since the $n$ eggs are disjoint, the hitting number is immediately computed as $h(\mathcal{S})=n$.  We claim that $e(\mathcal{S})=3m$. Since the set of edges incident to a single column forms an egg-cut, we have $e(\mathcal{S})\leq 6m$. To argue that $e(\mathcal{S})\geq 3m$, we find $6m$ pairwise edge-disjoint paths between an arbitrary pair of eggs.  Starting from one egg, travel to the left from any vertex horizontally, or zigzagging between $(1,1)$ and $(1,-1)$ moves (starting with either) until the other egg is reached.  This gives $3m$ paths.  Another $3m$ paths can be found by travelling to the right.  Any egg-cut separating those two eggs must include an edge from each of the $6m$ paths.  Since we chose an arbitrary pair of eggs, we have $e(\mathcal{S})\geq 6m $, and thus $e(\mathcal{S})=6m$.  This gives us
\[6m=\min\{n,6m\}=||\mathcal{S}||\leq \sn(K^t_{m\times n})\leq \gon(K^t_{m\times n})\leq 6m,\]
implying $\gon(K^t_{m\times n})=6m$.


    \begin{figure}[hbt]
        \centering
        \includegraphics[scale=0.8]{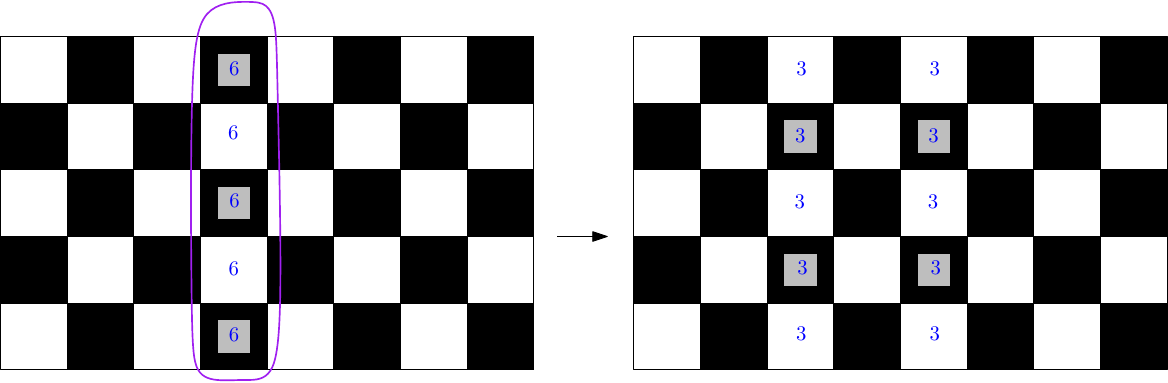}
        \caption{A positive rank divisor on a portion of a $5\times n$ toroidal king's graph. Firing the circled subset on the first divisor yields the second divisor.} 
        \label{fig:tor-king}
    \end{figure}
\end{proof}

We remark that we do not have $\gon(K^t_{m\times n})=6m$ for every choice of $m$ and $n$ with $3\leq m \leq n$.  For example, $K^t_{3\times 3}$ is a complete graph on $9$ vertices, and so has gonality $8$.  We can also compute $K^t_{3\times 4}$ and $K^t_{3\times 5}$ as each vertex has valence strictly larger than half the number of vertices. 
 This implies by Lemma \ref{lemma:high_valence} that $\gon(K^t_{3\times 4})=12-\alpha(K^t_{3\times 4})=10$ and $\gon(K^t_{3\times 5})=15-\alpha(K^t_{3\times 5})=13$. 
 More generally, the upper bound of
\[\gon(K^t_{m\times n})\leq mn-\alpha(K^t_{m\times n})=\begin{cases}mn-
\left\lfloor\frac{m}{2}\right\rfloor \left\lfloor \frac{n}{2}\right\rfloor\textrm{ if $m$ or $n$ is even}
\\ mn-\left \lfloor \frac{n}{2}\left\lfloor \frac{m}{2}\right\rfloor\right\rfloor\textrm{ if $m$ and $n$ are both odd}\end{cases}\]
gives a better upper bound than $6m$ when $m\leq n\leq 7$.

\section{Bishop's Graphs}
\label{section:bishops_graphs}

We now turn to bishop's graphs.
For $n,m\geq 2$, note that $B_{m\times n}$ is disconnected, with exactly two components: one representing white tiles on a chess board, and the other representing black tiles. We denote these two components $B_{m\times n}^w$ and $B_{m\times n}^b$, and take $B_{m\times n}^w$ to contain the vertex $(1,1)$. We remind the reader that for disconnected graphs, gonality can be computed as the sum of the gonalities of the disconnected components, so $\gon(B_{m\times n})=\gon(B_{m\times n}^w)+\gon(B_{m\times n}^b)$.  Since $B_{2\times in}$ consists of two path graphs, each with gonality $1$ by Lemma \ref{lemma:common_graph_gonalities}, we have $\gon(B_{2\times n})=2$.

Our general upper bound on $\gon(B_{m\times n})$ with $3\leq m\leq n$ will depend solely on $m$.  Before stating this result, we consider $3\times n$,  $4\times n$, and $5\times n$ bishop's graphs.  This will allow us to build up a general strategy for the $m\times n$ case.

\begin{example}

Consider the divisor $D$ on $B_{3\times n}$ illustrated on the top left in Figure \ref{figure:bishops_constant}.  Firing the first two columns transforms $D$ into $D'$, illustrated to the right.  Note that $D'$ appears identical to $D$, except with each chip shifted to the right by one unit.  Firing the first three columns, then the first four, and so on continues this pattern, translating the chips across the graph until there are chips in the final column.  This places chips on almost every vertex, with two omitted, namely $(2,1)$ and $(2,n)$.  These can still be covered:  $(2,n)$ by repeating our ``fire the first $k$'' columns strategy, and $(2,1)$ by symmetry.  Thus this divisor has positive rank, regardless of the value of $n$.  It follows that $\gon(B_{3\times n})\leq 12$ for all $n$.

\begin{figure}[hbt]
    \centering
    \includegraphics[scale=0.7]{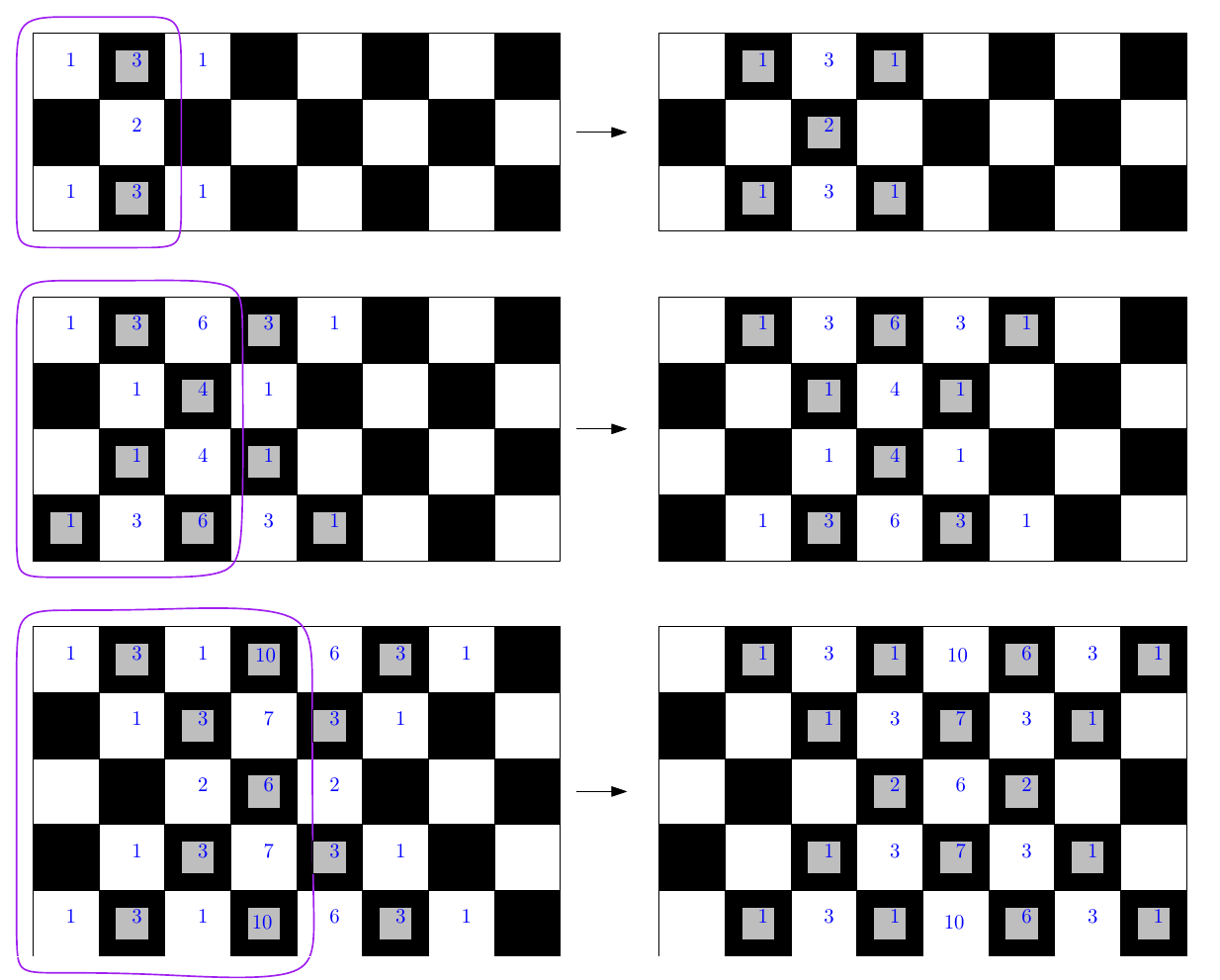}
    \caption{Divisors on $B_{3\times n}$, $B_{4\times n}$, and $B_{5\times n}$.  Regardless of how large $n$ is, these divisors have positive rank.}
    \label{figure:bishops_constant}
\end{figure}

Divisors are also illustrated in Figure \ref{figure:bishops_constant} for $B_{m\times n}$ with $m=4$ and $m=5$.  Firing the first $m-1$ columns translates the divisor horizontally as illustrated, and we can once again use this to argue that the divisor will have positive rank regardless of the value of $n$.  Summing up the number of chips, we find $\gon(B_{4\times n})\leq 40$ and $\gon(B_{5\times n})\leq 100$.

Although these chip placements are clearly suboptimal for small values of $n$, the key takeaway is that there is a universal bound for each choice of $m\in\{3,4,5\}$ that is completely independent of $n$. We will generalize this to all $m\geq 3$ in the following theorem, but first we explain intuitively how to build these divisors.  Chips are placed in the first $m-1$ columns, with a mirror image of the first $m-2$ columns in columns $m$ through $2m-3$.  To determine how many chips a vertex $(i,j)$ in the first $m-1$ columns should receive, we determine how many chips it will lose as we fire the first $m-1$ columns, then the first $m$ columns, and so on.  For instance, the vertex $(2,4)$ in $B_{5\times n}$ will lose $1$ chip up and to the right, and $3$ chips down and to the right on the first firing move; then $2$ chips down and to the right on the second firing move; then $1$ chip down and to the right on the third firing move, for a total of $1+(3+2+1)=7$ chips.  Due to the structure of the graph, this number will always be the sum of two triangular numbers $T_k:=\sum_{\ell=1}^k \ell$; in our example, we have $D(2,4)=T_1+T_3$.  Choosing this many chips ensures that no vertex will go into debt through our firing process; and as we will see in the proof of the following theorem, the recursive nature of triangular numbers leads to the translation patterns we have observed thus far.
\end{example}

Recall that $T_k=\frac{k(k+1)}{2}$ for $k\geq 1$; by convention, we take $T_k=0$ for $k\leq 0$.  Note that triangular numbers satisfy the relation $T_{k-1}+\max\{0,k\}=T_k$ for any integer $k$.

\begin{theorem}\label{theorem:bishops_general}
    For $3\leq m\leq n$, we have $\gon(B_{m\times n})\leq \frac{m^2(m^2-1)}{6}$.
\end{theorem}

\begin{proof}    Without loss of generality, we may assume that $n\geq 2m-3$; otherwise we may use the divisor that places one chip on every vertex, which has degree smaller than $\frac{m^2(m^2-1)}{6}$.  This ensures we will have enough horizontal space to build our divisor.

Define $D(i,j)$ on $B_{m\times n}$ by
\[\begin{cases} D(i,j)=T_{j-i+1}+T_{i+j-m}&\text{ for } j\leq m-1\\
D(i,j)=D(i,2(m-1)-j)&\text{ for }j\geq m.\end{cases}
\]

We first show that $D$ has positive rank.  Consider firing the first $m-1$ columns of the graph, transforming $D$ into $D'$.  If $j\leq m-1$, then the number of chips the vertex $(i,j)$ loses up and to the right is $\max\{0,i+j-m\}$, and the number of chips it loses down and to the left is $\max\{0,j-i+1\}$.  In other words, the number of chips it loses is $\max\{0,i+j-m\}+\max\{0,j-i+1\}$.  This means that
\begin{align*}
   D'(i,j)=& D(i,j)-\max\{0,i+j-m\}-\max\{0,j-i+1\}
   \\=& T_{j-i+1}+T_{i+j-m}-\max\{0,i+j-m\}-\max\{0,j-i+1\}
    \\=& (T_{j-i+1}-\max\{0,j-i+1)+(T_{i+j-m}-\max\{0,i+j-m\})\}
   \\=&T_{j-i}+T_{i+j-m-1}
   \\=&D(i,j-1)
\end{align*} for $i\leq m-1$.  If $j=m$, the vertex $(i,m)$ starts with $T_{m-i-1}+T_{i-2}$ chips and gains $m-1$ chips, meaning $D'(i,m)=T_{m-i-1}+T_{i-2}+(m-1)=T_{m-i-1}+(m-i)+T_{i-2}+(i-1)=T_{m-i}+T_{i-1}=D(i,m-1)$.
Finally, for $j> m$, the vertex $(i,j)$ will gain a number of chips equal to $\max\{0,2m-i-j\}+\max\{0,m+i-j-1\}$.  It follows that 
\begin{align*}
   D'(i,j)=& D(i,j)+\max\{0,2m-i-j\}+\max\{0,m+i-j-1\}
   \\=& D(i,2(m-1)-j)+\max\{0,2m-i-j\}+\max\{0,m+i-j-1\}
   \\=& T_{2(m-1)-j-i+1}+T_{i+2(m-1)-j-m}+\max\{0,2m-i-j\}+\max\{0,m+i-j-1\}
   \\=&(T_{2m-j-i-1}+\max\{0,2m-i-j\})+(T_{m+i-j-2}+\max\{0,m+i-j-1\})
   \\=& T_{2m-j-i}+T_{m+i-j-1}
   \\=& D(i,2(m-1)-(j-1))
   \\=&D(i,j-1).
\end{align*} 
In all cases, we have $D'(i,j)=D(i,j-1)$.

Thus firing the first $m-1$ columns has the net effect of translating the divisor one unit to the right. Firing larger and larger subsets of columns repeats this process, translating the divisor along the whole graph. This will eventually place a chip on every vertex, with the exception of two triangular wedges, one on the left and one on the right. However, continuing the process of firing larger and larger sets of columns will eventually place chips on all the vertices in the right wedge; and a symmetric argument handles the vertices in the left wedge. Thus the divisor $D$ has positive rank, as desired.

We now compute $\deg(D)$. In the $i^{th}$ row, the term $T_{j-i+1}$ yields the nonzero triangular numbers \[T_1,\ldots,T_{m-i-1},T_{m-i},T_{m-i-1},\ldots T_1,\] while the term $T_{i+j-m}$ yields the nonzero triangular numbers \[T_1,\ldots,T_{i-2},T_{i-1},T_{i-2},\ldots,T_1.\] It follows that the total number of chips in the $i^{th}$ row is
\[\sum_{k=1}^{m-i} T_k+\sum_{k=1}^{m-i-1} T_k+\sum_{k=1}^{i-1} T_k+\sum_{k=1}^{i-2} T_k,\]
which can be rewritten as
\[\frac{1}{6}\left[(m-i)(m-i+1)(m-i+2)+(m-i-1)(m-i)(m-i+1)+(i-1)i(i+1)+(i-2)(i-1)i\right].\]
Summing this expression as $i$ goes from $1$ to $m$ yields $\frac{1}{6}m^2(m^2-1)$ for the degree of $D$. Since $D$ has positive rank, we have that $\gon(B_{m\times n})\leq \frac{1}{6}m^2(m^2-1)$, as desired.
\end{proof}

It is natural to ask whether this bound is sharp, or at least of the right order of magnitude in terms of $m$. Unfortunately, the tool of scramble number will not suffice to answer this question, as detailed below.

\begin{proposition}For $n\geq \frac{1}{6}(m-1)m^2(m+1)$ and $m\geq 4$ even, we have $\sn(B^w_{m\times n})=\sn(B^b_{m\times n})=\frac{1}{6}(m-1)m(m+1)$.
\end{proposition}

\begin{proof}
    To   show $\sn(B^w_{m\times n})\geq \frac{1}{6}(m-1)m(m+1)$, we consider the scramble $\mathcal{S}$ on  $B^w_{m\times n}$ whose eggs are disjoint and consist of the first $m$ columns, then the next $m$ columns, and so on, with the last egg receiving between $m$ and $2m-1$ columns.  Since the eggs are disjoint, we have $h(\mathcal{S})=|\mathcal{S}|=\lfloor n/m\rfloor$.

    We now construct a collection of pairwise edge-disjoint paths on $B^w_{m\times n}$, which will be used to lower bound $e(\mathcal{S})$.  Our strategy is as follows: choose an integer $k$ with $1\leq k\leq m-1$. Choose a vertex in the first $k$ columns and the first $m-k$ rows of $B^w_{m\times n}$. From this vertex, alternate $(-k,k)$ and $(k,k)$ moves (starting with $(-k,k)$) until no more moves are possible. Note that there are $k(m-k)$ vertices in the first $k$ columns and the first $m-k$ rows, of which $\lceil k(m-k)/2\rceil$ are white. Thus we have constructed
    \[\sum_{k=1}^{m-1}\lceil k(m-k)/2\rceil\]
    ``down-paths'', which are edge-disjoint by construction.  We may similarly construct a collection of ``up-paths'' for each $k$, choosing a vertex in the first $k$ columns and the last $m-k$ rows, alternating between $(k,k)$ and $(-k,k)$ moves (starting with $(k,k)$) until no more moves are possible. A similar analysis gives us  
        \[\sum_{k=1}^{m-1}\lfloor k(m-k)/2\rfloor\]
        ``up-paths''.  Taken together, we have constructed a family of
    \[\sum_{k=1}^{m-1}\left(\lceil k(m-k)/2\rceil+\lfloor k(m-k)/2\rfloor\right)=\sum_{k=1}^{m-1} k(m-k)=\frac{(m-1)m(m+1)}{6}\]
    pairwise edge-disjoint paths on $B^w_{m\times n}$.

    Since the eggs of $\mathcal{S}$ are $m\times m$ subgraphs and all moves go forward by at most $m-1$, every one of our paths intersects every egg.  Thus by truncating paths appropriately, for every pair of eggs we can find \(\frac{(m-1)m(m+1)}{6}\) pairwise edge-disjoint paths connecting them.  It follows that $e(\mathcal{\mathcal{S}})\geq \frac{(m-1)m(m+1)}{6}$.  On the other hand, $e(\mathcal{\mathcal{S}})\leq \frac{(m-1)m(m+1)}{6}$; for instance, this equals the number of edges connecting the first two eggs.

    Taken together, we have $||\mathcal{S}||=\min\{h(\mathcal{S}),e(\mathcal{S})\}=\min\{\lfloor n/m\rfloor,\frac{1}{6}(m-1)m(m+1)\}=\frac{1}{6}(m-1)m(m+1)$.  It follows that $\sn(B^w_{m\times n})\geq \frac{1}{6}(m-1)m(m+1)$.

    For the upper bound we construct a tree-cut decomposition $\mathcal{T}$ on $B^w_{m\times n}$.  Place the first $m$ columns in one bag, then the next $m$, and so on, with the last bag receiving possibly fewer than $m$ columns; connect the bags in a path from left to right.  Each bag has at most $m^2/2$ vertices, and there are no tunneling edges.  The number of edges in each link is at most $\frac{1}{6}(m-1)m(m+1)$: indeed, every edge between a pair of adjacent bags is part of one of the down-paths or up-paths we considered previously. Every path has exactly one edge connecting two adjacent bags with $m$ columns each, so at least one link has exactly $\frac{1}{6}(m-1)m(m+1)$ edges.  Thus the width of the tree-cut decomposition is $\max\{m^2/2,\frac{1}{6}(m-1)m(m+1)\}=\frac{1}{6}(m-1)m(m+1)$, where we use the fact that $m\geq 4$.  This gives us $\sn(B^w_{m\times n})\leq \scw(B^w_{m\times n})\leq w(\mathcal{T})=\frac{1}{6}(m-1)m(m+1)$, implying that $\sn(B^w_{m\times n})=\frac{1}{6}(m-1)m(m+1)$.

    An identical argument shows that $\sn(B^b_{m\times n})=\frac{1}{6}(m-1)m(m+1)$.
\end{proof}

Since $B^w_{(m-1)\times n}$ is a subgraph of $B^w_{m\times n}$ which is a subgraph of $B^w_{(m+1)\times n}$, this immediately gives us the upper and lower bounds of
\[\frac{1}{6}(m-2)(m-1)m\leq \sn(B^w_{m\times n})\leq \frac{1}{6} m(m+1)(m+2)\]
for $m$ odd; a similar bound can be found for $\sn(B^b_{m\times n})$.  
This means that scramble number provides us with an $O(m^3)$ lower bound on $\gon(B_{m\times n})$, as detailed below; however, no stronger lower bound can be proved using scramble number.

\begin{corollary}
    Let $m\geq 4$, $n\geq \frac{1}{6}(m-1)m^3(m+1)$, and $m'=2\lfloor m/2\rfloor$.  We have $\gon(B_{m\times n})\geq \frac{1}{3}(m'-1)m'(m'+1)$.
\end{corollary}

\begin{proof} First assume $m$ is even, so $m'=m$.  Then 
\[\gon(B_{m\times n})=\gon(B^w_{m\times n})+\gon(B^b_{m\times n})\geq \sn(B^w_{m\times n})+\sn(B^b_{m\times n})=\frac{1}{3}(m'-1)m'(m'+1)\]
by the previous result.

If $m$ is odd, then $m'=m-1$. Since $B^w_{m\times n}$ has $B^w_{m'\times n}$ as a subgraph, we know that $\sn(B^w_{m\times n})\geq \sn(B^w_{m'\times n})=\frac{1}{6}(m'-1)m'(m'+1)$; similarly, $\sn(B^b_{m\times n})\geq \frac{1}{6}(m'-1)m'(m'+1)$.  This gives us the claimed lower bound on  $\gon(B_{m\times n})$.
\end{proof}

The last avenue we might take to evaluate the bounds from Theorem \ref{theorem:bishops_general} is from a computational perspective.  This is unfeasible for $m\geq 4$, since the upper bound of $mn-\alpha(B_{m\times n})$ will outperform our bound until $n$ is quite large, since $\alpha(B_{m\times n})\in\{m+n-2,m+n-1\}$ for $m<n$ \cite[Proposition 7.3]{bishops_independence}.  For $m=4$, for instance, this upper bound is at least $4n-(4+n-1)=3n-3$, which outperforms the upper bound of $40$ as long as $n\leq 14$.  The first graph where our result would do better is $B_{4\times 15}$, whose components each have $30$ vertices.  However, for $m=3$ (where we have an upper bound of $12$ on gonality), we find using the Chip-Firing Interface \cite{chip_firing_interface} that $\gon(B_{3\times 10}^w)=\gon(B_{3\times 10}^b)=6$, so $\gon(B_{3\times 10})=12$ (for smaller choices of $n$, the gonality is strictly lower than $12$).  It is worth noting that gonality can be larger for a subgraph than a larger graph (even if both graphs are connected), so this computation does not necessarily imply that $\gon(B_{3\times n})\geq 12$ for $n>10$.

\section{Toroidal Bishop's Graphs}
\label{section:toroidal_bishops_graphs}
We now consider the structure of $B_{m\times n}^t$, the $m\times n$ toroidal bishop's graph. Following the work in \cite{torBishop}, we classify the diagonals as 
$s$-diagonals and $d$-diagonals, where $d$-diagonals are those that go to the right and down, whereas $s$-diagonals go to the right and up. Several $s$- and $d$-diagonals are illustrated in Figure \ref{fig:s-d-diag}. There are several features that distinguish the behavior of these diagonals on a toroidal versus a non-toroidal bishop's graph:
\begin{itemize}
    \item On a non-toroidal graph, $s$-diagonals and $d$-diagonals have length at most $\min\{m,n\}$; on a toroidal graph, they can be longer.
    \item On a non-torodial graph, an $s$-diagonal and a $d$-diagonal can intersect at most once; on a toroidal graph, they can intersect more than once.
    \item On a non-toroidal graph, diagonals are monochromatic (consisting entirely of black squares, or of white squares); on a toroidal graph, if either $m$ or $n$ is odd, $s$-diagonals and $d$-diagonals have multiple colors.
\end{itemize}
We present a few lemmas to better understand the structure of $B_{m\times n}^t$ in terms of these diagonals.


\begin{figure}[hbt]
    \centering
    \includegraphics[width = 0.7\textwidth]{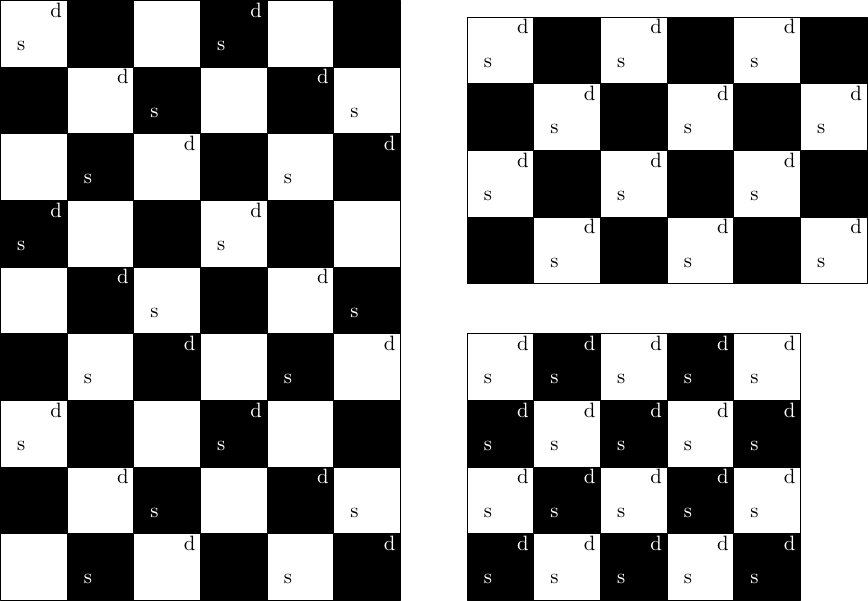}
    \caption{The $s$- and $d$-diagonals beginning at the top left vertex of the graphs $B_{9 \times 6}^t, B_{4 \times 6}^t$, and $ B_{5 \times 6}^t$. 
    }
    \label{fig:s-d-diag}
\end{figure}

\begin{lemma}[\cite{torBishop}]
    The length of each diagonal of $B_{m\times n}^t$ is $\lcm(m,n)$.
\end{lemma}
The summary of the proof from \cite{torBishop} is quick to give:  a bishop moving along a fixed diagonal will return to its row every $n$ moves, and to its column every $m$ moves; therefore it will first return to its initial location after $\lcm(m,n)$ moves.

Since we may partition $V(B^t_{m\times n})$ into either $s$-diagonals or $d$-diagonals, this immediately leads to the following corollaries.

\begin{corollary}
     There are $\gcd(m,n)$ many $s$-diagonals, and $\gcd(m,n)$ many $d$-diagonals.
\end{corollary}

\begin{corollary}
The graph $B_{m\times n}^t$ is isomorphic to the complete graph $K_{mn}$ if and only if $m$ and $n$ are relatively prime.
\end{corollary}

We now describe the intersection properties of $s$-diagonals with $d$-diagonals.

  \begin{lemma} For a toroidal bishop's graph $B_{m\times n}^t$,
~\begin{itemize}
    \item[(i)] if $m$ or $n$ is odd, every $s$-diagonal intersects every $d$-diagonal at least once; and
    \item[(ii)] if $m,n$ are both even, every $s$-diagonal intersects every $d$-diagonal of its color at least once.
\end{itemize}
\end{lemma}
\begin{proof}

   For claim (i) we assume at least one of $m$ and $n$ is odd.
    Without loss of generality, consider the $d$-diagonal that contains all $(a,b)$ such that $a-b\equiv 0 \mod\gcd(m,n)$. Now, consider any $s$-diagonal; for some $k$, it contains all $(a,b)$ such that $a+b\equiv k \mod\gcd(m,n)$. These diagonals intersect if there is some point in common between the two; that is, if there exists $a$ such that $2a\equiv k\mod \gcd(m,n)$. But this is indeed the case: $\gcd(m,n)$ is odd, so for any $k$, either $k$ or $\gcd(m,n)+k$ is even. Therefore dividing $k$ or $\gcd(m,n)+k$ by 2 yields an integer; it follows that a diagonal starting on $(0,a)$ or $(a,0)$ will intersect the desired $d$-diagonal. 

  For claim (ii) we assume $m$ and $n$ are both even.  The black and white components of the graph are isomorphic, so without loss of generality, consider again the diagonal of $(a,b)$ with $a-b\equiv 0\mod \gcd(m,n)$. The only $s$-diagonals of the same color are of the form $(a,b)$ with $a+b\equiv k \mod \gcd(m,n)$ for some even $k$. So at intersections with the fixed $d$-diagonal, $2a\equiv k\mod \gcd(m,n)$. Since $k$ is even, such an $a$ exists, so the desired intersection exists.
\end{proof}

We are now ready to compute the gonality of all toroidal bishop's graphs.  Our proof is inspired by that of \cite[Theorem 1.1]{rooks_gonality}, where the gonality of rook's graphs was determined.

\begin{theorem} For $m,n\geq 2$, we have $\gon(B^t_{m\times n}) = mn-\gcd(m,n)$\end{theorem}
\begin{proof}
First we note that $\alpha(B^t_{m\times n})=\gcd(m,n)$ by \cite[\S 6]{torBishop}.  
Since  $B^t_{m\times n}$ has no isolated vertices, we may apply Lemma \ref{lemma:n-alpha(G)} to deduce $\gon(B^t_{m\times n}) \leq |V(B^t_{m\times n})|-\alpha(B^t_{m\times n})=mn-\gcd(m,n)$. 

It remains to show that $\gon(B^t_{m\times n}) \geq mn-\gcd(m,n)$. First, consider the case where $m,n$ are not both even, so that the bishop's graph is connected.
Suppose there is an effective divisor $D_0$ of positive rank with degree strictly less than $mn-\gcd(m,n)$. Among all effective divisors equivalent to $D_0$, let $D$ be one where the $d$-diagonal with the fewest chips placed on it has the most chips. There are $\gcd(m,n)$ many diagonals, and $\frac{mn-\gcd(m,n)}{\gcd(m,n)} = \lcm(m,n)-1$, so if there are fewer chips than $mn-\gcd(m,n)$ then at least one $d$-diagonal has fewer chips than $\lcm(m,n)-1$. Now,choose a vertex $q$ on the poorest $d$-diagonal with $D(q)=0$, and run Dhar's burning algorithm. Since the $d$-diagonal is isomorphic to the complete graph $K_{\lcm(m,n)}$, the fire burns the whole diagonal since there are fewer than $\lcm(m,n)-1$ chips \cite[Proposition 14]{aidun2019gonality}. Since $m,n$ are both odd, this diagonal intersects every other diagonal at least once, and therefore starts a fire along each diagonal.  Since the poorest $s$-diagonal also has strictly less than $\lcm(m,n)-1$ chips, and it also burns.

Since $D$ has positive rank and $q$ has no chips, we know that the whole graph does not burn.  Let $U$ be the set of vertices that do not burn, which Dhar's algorithm dictates we now fire to obtain a new divisor $D'$. Since the burned $s$-diagonal intersects each $d$-diagonal, we know that each has at most $\lcm(m,n)-1$ vertices in $U$. Thus when $U$ is fired, any $d$-diagonal with vertices in $U$ moves $k(\lcm(m,n)-k)\geq \lcm(m,n)-1$ chips to the burned vertices of that diagonal. Therefore each $d$-diagonal that has unburned vertices has more chips than the richest poorest diagonal we started with.  On the other hand every $d$-diagonal that burned completely will gain chips, as the unburned portions of $d$-diagonals fire along $s$-diagonals. Thus $D'$ has a richer poorest $d$-diagonal than $D$, a contradiction.  Thus any positive rank divisor on $B^t_{m\times n}$ has degree at least $mn-\gcd(m,n)$ when at least of of $m$ and $n$ is odd.

If $m,n$ are both even, the gonality of the graph is simply the sum of the gonalities of its components. Thus if the gonality is less than $nm-\gcd(nm)$, one of the connected components must have gonality less than $\frac{nm-\gcd(nm)}{2}$. Writing $n=2a,m=2b$, the gonality of that component must be less than $2ab - 2\gcd(a,b)$. So, some $d$-column has fewer than $\frac{2ab - 2\gcd(a,b)}{2\gcd(a,b)} = \lcm(a,b)-1$ chips for any equivalent divisor. From there, the argument then proceeds as before.
\end{proof}

The reader may wonder whether this result could be proved by means of scramble number.  For certain values of $m$ and $n$, the answer is yes; for instance, if $\gcd(m,n)=1$, then we have a complete graph, which has scramble number equal to gonality.  For other values, this is not the case, even if the graph is connected.

\begin{figure}[hbt]
    \centering
    \includegraphics{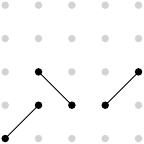}
    \caption{Six vertices on the $5\times 5$ toroidal bishop's graph.}
    \label{figure:toroidal_bishop_5x5}
\end{figure}

\begin{example}
In this example we show that the $5\times 5$ toroidal bishop's graph has scramble number strictly smaller than gonality.  By the previous theorem, we know $\gon(B^t_{5\times 5})=5\cdot 5-\gcd(5,5)=25-5=20$.  Suppose for the sake of contradiction that $\sn(B^t_{5\times 5})=20$, and let $\mathcal{S}$ be a scramble of order $20$ on $B^t_{5\times 5}$.  Consider the $19$ grey vertices in Figure \ref{figure:toroidal_bishop_5x5}.  Since $h(\mathcal{S})\geq ||\mathcal{S}||=20$, these do not form a hitting set, so there must be an egg $E_1$ contained among the other six vertices.   The only edges shared by the six black vertices are the three pictured, so the egg must consist of a single vertex or a pair of adjacent vertices.  Moreover, since those six vertices also cannot form a hitting set for $\mathcal{S}$, there is a second egg $E_2$ contained among the other $19$. Deleting the edges connecting $E_1$ to the rest of the graph thus forms an egg-cut.  If $E_1$ is a single vertex, the egg-cut has size $8$; and if $E_1$ is a pair of adjacent vertices, the egg-cut has size $8+8-2=14$.  Either way, we have $||\mathcal{S}||\leq e(\mathcal{S})\leq 14<20$, a contradiction.  Thus $\sn(B^t_{5\times 5})< \gon(B^t_{5\times 5})$.
\end{example}

\section{Knight's Graphs}\label{section:knightsgraphs}

The story for knight's graphs is similar to bishop's graphs:  we can find a bound on $\gon(N_{m\times n})$ depending solely on $m$, although this bound is provably suboptimal for graphs with $m$ and $n$ close to one another.   We begin by presenting our most general upper bound, and then pivot to more intermediate cases.

\begin{theorem}\label{theorem:knights_most_general}
    For the knight's graphs $N_{m\times n}$ with $m\leq n$, we have the following formulas and  bounds:
    \begin{enumerate}
        \item $\gon(N_{2\times n})= 4$.
        \item $\gon(N_{3 \times n}) \leq 18$
        \item $\gon(N_{m \times n}) \leq 10m - 12$ for $m\geq 4$.
    \end{enumerate}
\end{theorem}
\begin{proof}[Proof]
We remark that $N_{2 \times n}$ has four connected components, each a path, as illustrated in Figure \ref{2_by_n_knight's_proof_figure}.  A path has gonality $1$, so $\gon(N_{2 \times n})=4$.

    \begin{figure}[hbt]
        \centering
        \includegraphics[width = 0.3\textwidth]{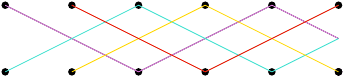}
        \caption{The graph $N_{2\times6}$, whose four components are all paths.}
        \label{2_by_n_knight's_proof_figure}
    \end{figure}

For $m\geq 3$, we construct a divisor $D$ as follows.  Within the first and third columns, place one chip on the top and bottom vertices, and two chips on all other vertices.  Within the second column, place three chips on the top and bottom vertices.  Then, if $m=3$, place $4$ chips on the middle vertex of that column; and if $m\geq 4$, place $5$ chips on the vertices directly above the top vertex and directly below the bottom vertex, and $6$ chips on all remaining vertices in the column.  This divisor $D$ is illustrated for $3\leq m\leq 5$ and $n=8$ on the left in Figure \ref{fig:const-knight}.  Note that if $m=3$, we have $\deg(D)=18$; and if $m\geq 4$,
\[\deg(D)=1+1+2(m-2)+3+3+5+5+6(m-4)+1+1+2(m-2)=10m-12.\]  In both cases  equals our claimed upper bound for $\gon(N_{m\times n})$, so to prove the remainder of our claim it will suffice to show that $D$ has positive rank.

\begin{figure}[hbt]
    \centering
    \includegraphics[width = 0.7\textwidth]{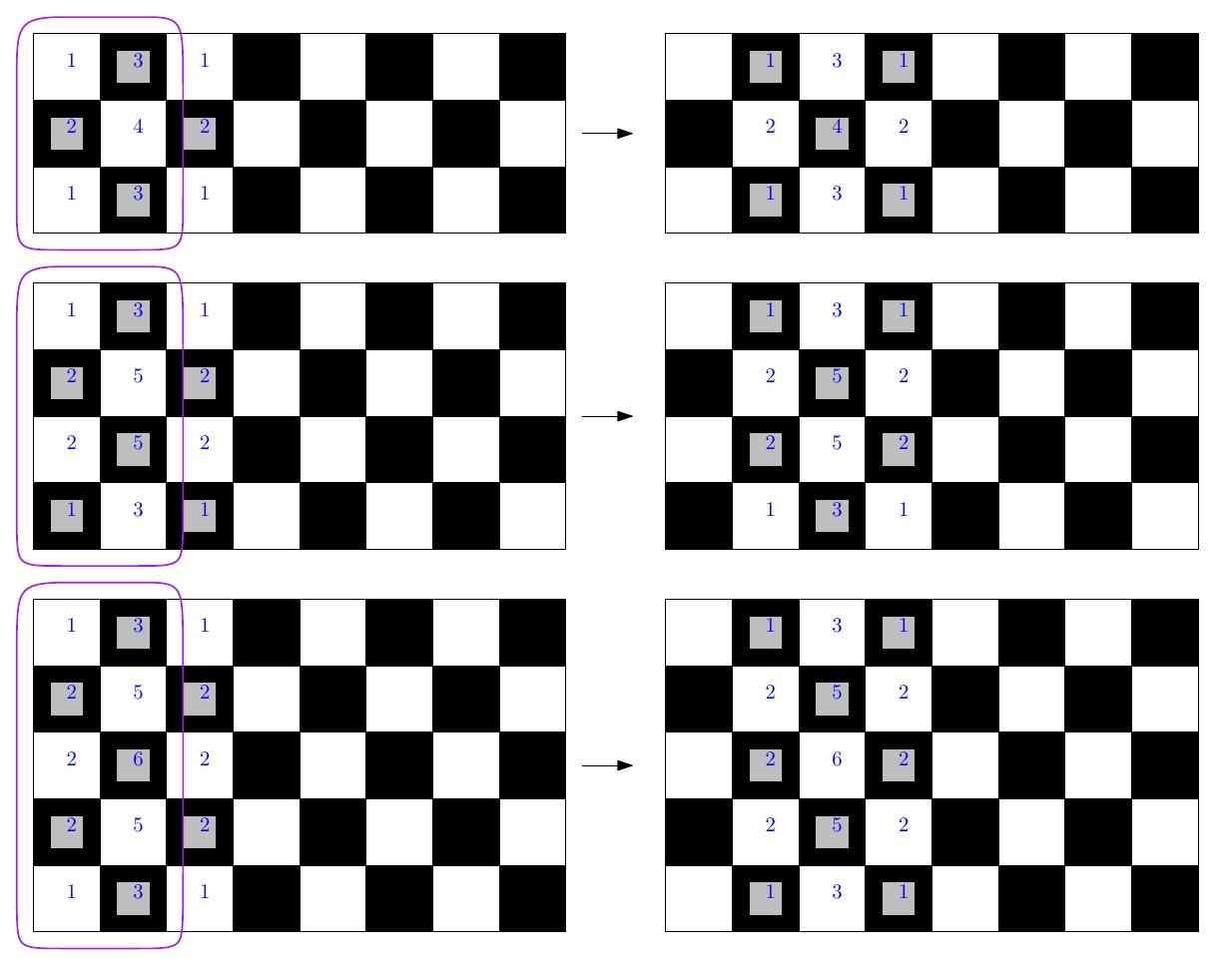}
    \caption{The divisor $D$ on $N_{3\times8}$, on $N_{4\times8}$, and on $N_{5\times8}$, together with the effect of firing the first two columns.}
    \label{fig:const-knight}
\end{figure}

Consider the effect on $D$ of firing the set $S$ of vertices in the first two columns of $N_{m\times n}$.  This will effect the first four columns of the graph as follows:
\begin{itemize}
    \item  In the first column, every vertex has as many neighbors outside $S$ as it has chips.  Thus every vertex loses all its chips.
    \item  In the second column, the top and bottom vertex have two neighbors  outside of $S$, so go from $3$ chips to $1$ chip. If $m=3$, the middle vertex has two neighbors outside of $S$, and so goes from $4$ chips to $2$ chips.  If $m\geq 4$, then the second topmost and second bottom-most vertices have $3$ neighbors outside of $S$, and so go from $5$ chips to $2$ chips; and any remaining vertex has $4$ neighbors outside of $S$, and so goes from $6$ chips to $2$ chips.  In summary, the second column becomes identical to the first column in $D$.
    \item  In the third column, the top and bottom vertex have $2$ neighbors in $S$, and so go from $1$ chip to $3$ chips.  If $m=3$, the middle vertex has $2$ neighbors in $S$, and so goes from $2$ chips to $4$ chips.  If $m\geq 4$, the second top-most and second bottom-most vertices have $3$ neighbors in $S$, and so go from $2$ chips to $5$ chips; and all other vertices have $4$ neighbors in $S$, and so go from $2$ chips to $6$ chips.  In summary, the third column becomes identical to the second column in $D$.
    \item Finally, in the third column, the top and bottom vertices each have $1$ neighbor in $S$, and all other vertices have $2$ neighbors.  Each vertex gains that many chips, and as each vertex in that column started with $0$ chips, it has become identical to the third column in $D$.
\end{itemize}
Thus the net effect of firing the first two columns is to translate all chips from $D$ to the right one vertex, as illustrated in Figure \ref{fig:const-knight}.  From here, set-firing the first three columns moves the chips to the right one vertex again.  We may continue in this fashion, set-firing the first $k$ columns for increasing values of $k$, until we have translated our chips all the way across the graph.  Each of the intermediate divisors is effective, and for every vertex in the graph at least one divisor places a chip on that vertex.  Thus $D$ has positive rank.
\end{proof}






The bounds from Theorem \ref{theorem:knights_most_general} with $m\geq 3$ are certainly not optimal for all $n$. For instance, since knight's graphs are bipartite, they have independence number at least $\lceil mn/2\rceil$, giving an upper bound on gonality of $mn-\lceil mn/2\rceil=\lfloor mn/2\rfloor$.  For $m=3$, this outperforms the bound of $18$ for $3\leq n\leq 11$.  The following theorem presents an alternate ``better-than-bipartite'' bound for $3\leq m\leq 5$.

\begin{theorem}\label{theorem:knights_3_and_4_and_5}  We have the following upper bounds on the gonality of knight's graphs.
\begin{itemize}
    \item[(i)] For $n\geq 4$, $\gon(N_{3\times n})\leq 2\left\lfloor \frac{n-2}{2} \right\rfloor$.
    \item[(ii)] For $n\geq 4$, $\gon(N_{4\times n})\leq 2n-4$.
    \item[(iii)] For $n\geq 5$,
$\gon(N_{5\times n}) \leq 2n+\left\lfloor\frac{n}{5}\right\rfloor$.
\end{itemize}
\end{theorem}

We reserve the proof of this result for Appendix \ref{appendix:knights}.  We remark that this theorem outperforms Theorem \ref{theorem:knights_most_general} for $m=3$ and $4\leq n\leq 19$; for $m=4$ and $4\leq n\leq 15$; and for $m=5$ and $5\leq n\leq 17$.  We do not expect the above to be the exact formula for gonality for all these values, but rather a demonstration that there exist strategies outperforming those of Theorem \ref{theorem:knights_most_general} for certain values of $m$ and $n$.  Still, it is worth asking whether these formulas are optimal for any values of $n$.  When $m=3$, Theorem \ref{theorem:knights_3_and_4_and_5} is optimal for $n\in\{4,5,6\}$:

\begin{proposition}\label{prop:knights_exact_3_by_small}  We have $\gon(N_{3\times 3})=3$, $\gon(N_{3\times 4})=2$, and $\gon(N_{3\times 5})=\gon(N_{3\times 6})=4$.
\end{proposition}

We prove this proposition in Appendix \ref{appendix:knights}.  Using the Chip-Firing Interface \cite{chip_firing_interface}, we also verify that Theorem \ref{theorem:knights_3_and_4_and_5} gives the correct gonality for $m=3$ and $7\leq n\leq 11$ and for $m=4$ and $4\leq n\leq 6$, although it is unable to compute gonality for larger knight's graphs.

We now pivot to lower bounds on gonality.  In contrast to bishop's graphs, we can use scramble number to find a lower bound on $\gon(N_{m\times n})$ with the same order of magnitude as our upper bound from Theorem \ref{theorem:knights_most_general}. 

\begin{proposition}\label{prop:knights_sn}  Let $m\geq 4$ and let $\lfloor n/3\rfloor \geq 6m-8$.  Then $\sn(N_{m\times n})= 6m-8$.
\end{proposition}

\begin{proof}  Construct a scramble $\mathcal{S}$ whose first egg consists of the first three columns of $N_{m\times n}$, the second egg the next three columns, and so on; if $n$ is not a multiple of $3$, then any leftover columns are included in the same egg as the last complete set of four columns.  Note that $N_{m\times 3}$ is connected since $m\geq 4$, so this is a valid scramble.

As the eggs are disjoint, we have $h(\mathcal{S})=|\mathcal{S}|=\lfloor n/3\rfloor$.  To lower bound the egg-cut number by $k$, we can construct $k$ pairwise edge-disjoint paths, each of which intersects every egg. First we consider paths stretching across the graph alternating between $(2,1)$ and $(2,-1)$ moves.  We can find $4m-4$ pairwise edge-disjoint such paths:  start at any vertex in the first two columns, and choose whether to start with a $(2,1)$ move or a $(2,-1)$ move.  The only restriction is that a vertex in the top row cannot perform a $(2,1)$ move, and a vertex in the bottom row cannot perform a $(2,-1)$ move, accounted for by the $-4$.  Then we find $2m-4$ more paths, starting from a vertex in the first column and alternating $(1,2)$ and $(1,-2)$ moves in some order.  There are $2m$ ways to choose a vertex and a starting move, with $4$ excluded since the top two and bottom two vertices have only one option for their first move.

This gives a total of  $4m-4+2m-4=6m-8$ pairwise edge-disjoint paths.  Any egg-cut must include at least one edge from each of these paths, so $e(\mathcal{S})\geq 6m-8$.  We conclude that $\sn(N_{m\times n})\geq ||\mathcal{S}||\geq \min\{\lfloor n/4\rfloor,6m-8\}=6m-8$.

For an upper bound, we consider the tree-cut decomposition $\mathcal{T}$ whose bags are pairs of adjacent columns (first and second columns in one bag, third and fourth column in the next, and so on), except that the last bag has three columns if $n$ is odd.  Arrange these bags in a path, from left to right.  There are no tunneling edges, so the contribution of each bag is at most $3m$; and between any two adjacent bags, every edge comes from one of the $6m-8$ paths above, and each such path contributes exactly one edge. Since $m\geq 4$, we have $6m-8\geq 3m$, and so the width of this tree-cut decomposition is $6m-8$.  Thus we have $\sn(N_{m\times n})\leq \textrm{scw}(N_{m\times n})\leq w(\mathcal{T})=6m-8$.  This completes the proof.
\end{proof}

Thus for $m\geq 4$ and $n$ much larger than $m$, we have $6m-8\leq \gon(N_{m\times n})\leq 10m-12$ since scramble number is a lower bound on gonality.  This means that our upper bound is, in the worst possible case, approximately a factor of $5/3$ larger than the actual gonality.  However, scramble number cannot be used to narrow this gap any further.

\section{Toroidal Knight's Graphs} 
\label{section:toroidal_knights_graphs}
We can also find a bound on the gonality of the $m\times n$ toroidal knight's graph solely in terms of $m$, where $m\leq n$.

\begin{theorem}\label{theorem:main_toroidal_knights}
    The gonality of the $m\times n$ toroidal knight's graph is bounded by\[\gon(N_{m\times n}^t)\leq\begin{cases}
        20&\text{ for } m = 2\\
        20m&\text{ for } m \geq 3.
    \end{cases}\]
\end{theorem}

\begin{proof}
    For the $m=2$ case, if $n\leq 5$ the claim is immediate as we can place a chip on every vertex. 
 For $n\geq 6$, start at any vertex on the top row and place, from left to right, $1$, $3$, $1$, $1$, $3$, and $1$ chip; duplicate this directly below on the bottom row.  The resulting divisor $D$ is illustrated in Figure \ref{fig:2xn-tor-knight}, along with several equivalent divisors.  Firing the middle four columns has the net effect of translating the four $1,3,1$ collections of chips outward.  Firing the middle six repeats this, and so on, allowing the divisor to eventually move chips to any vertex.  Thus $r(D)>0$, and $\gon(N^t_{2\times n})\leq \deg(D)=20$. 
    \begin{figure}[hbt]
        \centering
        \includegraphics{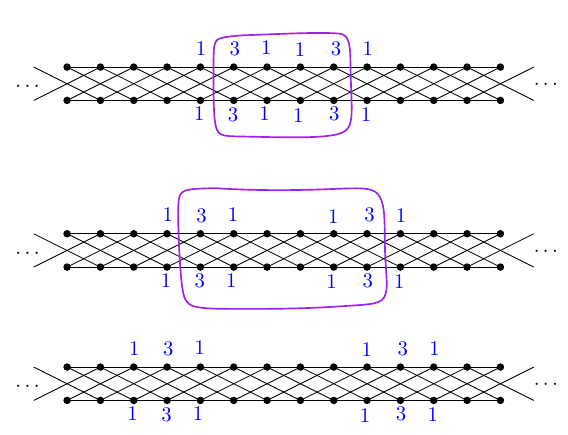}
        \caption{Three equivalent divisors on $N^t_{2\times n}$.}
        \label{fig:2xn-tor-knight}
    \end{figure}

For $m\geq 3$, a similar divisor on $N^t_{m\times n}$ can be built when $n\geq 6$ ($n\leq 5$ is handled by the small number of vertices), except with a pattern of $2$, $6$, $2$, $2$, $6$, $2$ replicated along the $m$ rows.  Again, set-firing the middle four columns, then the middle six, and so on translates chips along the graph, giving us a positive rank divisor of degree $20m$.
\end{proof}

As for knight's graphs, these divisors are provably suboptimal for small values of $n$. For choices of $m$ and $n$ with $mn<20m$, simply placing a chip on every vertex would yield a winning divisor of lower degree. Stronger bounds can be found when $m$ and $n$ are both small and even, as the graph is bipartite.  It remains an open question which, if any, toroidal knight's graphs have gonality equal to our upper bound.

As we did for traditional knight's graphs, we test the strength of our upper bound using scramble number.  

\begin{proposition} Let $m\geq 5$ and $\lfloor n/3\rfloor\geq 12m$.  Then $\sn(N^t_{m\times n})=12m$.
\end{proposition}

\begin{proof}  For the lower bound, we construct a scramble  $\mathcal{S}$ identical to that from the proof of Proposition \ref{prop:knights_sn}, with disjoint eggs consisting of $3$ columns each (with one egg receiving up to two more columns, if $n$ is not a multiple of $3$).  The hitting number of this scramble is given by the number of eggs: $h(\mathcal{S})=|\mathcal{S}|=\lfloor n/3\rfloor$.

To lower bound the egg-cut number, we consider an arbitrary pair of eggs, $E_1$ and $E_2$.  Starting in $E_1$, we construct $6m$ paths from $E_1$ to $E_2$ that travel to the right:  $4m$ starting from any vertex in the second and third column of $E_1$ and alternating $(2,1)$ and $(2,-1)$ moves in some order; and $2m$ starting from any vertex in the first column of $E_1$ and alternating $(1,2)$ and $(1,-2)$ in some order\footnote{To ensure these paths are different, we need our $m\geq 5$ assumption; if $m=4$, a $(1,2)$ move is identical to a $(1,-2)$ move.}.  We can similarly construct $6m$ paths from $E_1$ to $E_2$ travelling to the left:  $4m$ starting from any vertex in the first and second column of $E_1$ and alternating $(-2,1)$ and $(-2,-1)$ in some order; and $2m$ starting from any vertex in the first column of $E_1$ and alternating $(-1,2)$ and $(-1,-2)$ in some order.  These $12m$ paths are pairwise edge-disjoint, and any egg-cut separating $E_1$ and $E_2$ must include an edge from each. Thus we have $e(\mathcal{S})\geq 12m$, and so $\sn(N^t_{m\times n})\geq ||\mathcal{S}||\geq \min\{\lfloor n/3\rfloor, 12m\}=12m$.

For the upper bound on scramble number, we construct the exact same tree-cut decomposition $\mathcal{T}$ as in the non-toroidal case, letting our bags be groups of $2$ adjacent columns except for one final bag possibly consisting of $3$ columns.  The edges between adjacent bags increases from $6m-8$ to $6m$ due to the toroidal nature of the graph.  All edges from the leftmost bag to the rightmost bag are now tunneling edges; there are $6m$ of these, adding a weight of $6m$ to each link and each middle bag. Thus the width of this tree-cut decomposition is $12m$, giving $\sn(N^t_{m\times n})\leq \scw(G)\leq w(\mathcal{T})=12m$.  This completes the proof.
\end{proof}

Thus for $m\geq 5$ and $n$ much larger than $m$, we have $12m\leq \gon(N^t_{m\times n})\leq 20m$.  As with the non-toroidal knight's graph, we have at worst a factor of $5/3$ overestimate of gonality with our current upper bound; and no stronger results can be gleaned from scramble number.  We remark that for $m$ odd, we may modify our scramble to have eggs consisting of two adjacent columns (which form connected subgraphs), which allows us to loosen our restriction on $n$ to  $\lfloor n/2\rfloor \geq 12m$.

\section{Future Directions}\label{section:future_directions}

A major open question is whether our upper bounds on gonality are in fact equal to gonality for $m\times n$ chessboards, at least for $n$ sufficiently large compared to $m$.  The answer is yes for king's, toroidal king's, and toroidal bishop's, but remains open for bishop's, knight's, and toroidal knight's.  Such a result would not be provable using scramble number, so new lower bounds on gonality may need to be developed to prove such a result.

Another direction for future work could be modifying our chessboards, for instance by considering higher dimensional analogs of these graphs.  Even in $3$-dimensions, there is little known; for instance, the gonality the $m\times n\times \ell$ rook's graph $R_{m\times n\times l}$ is only known for \(\min\{m,n,\ell\}=2\) \cite[Theorem 1.3]{rooks_gonality}.  Another modification could be to the ``glueing'' rules; for instance, instead of toroidal chessboards, one could consider half-toroidal chessboards, M\"{o}bius chessboards, or Klein-bottle chessboards.

\bibliographystyle{plain}

\appendix

\section{Computations for king's graphs}
\label{appendix:kings}

\begin{lemma}\label{lemma:appendix_kings_scramble}
The scramble $\mathcal{S}$ on $K_{3\times 5}$ illustrated in Figure \ref{figure:scw_for_k34} satisfies $e(\mathcal{S})\geq 7$.   
\end{lemma}

\begin{proof}
    To prove this it will suffice to show that between any pair of eggs, there exist at least $7$ pairwise edge-disjoint paths between them (this is because any egg-cut separating a pair of eggs would have to include an edge from any such collection of paths).

     For pairs of eggs that are each a complete column, the argument from the proof of Theorem \ref{theorem:gonality_kings} gives $3\cdot 3-2=7$ such paths.  The remaining cases to handle are where the two eggs have the following numbers of vertices:  $1$ and $1$, $1$ and $2$, $1$ and $3$, $2$ and $2$, and $3$ and $3$.  Representative pairings of each type are shown in Figure \ref{figure:sample_eggs_k35}, with each row containing $7$ pairwise edge-disjoint paths between that pair of eggs.  These paths can be adapted for all unpictured pairs of eggs either using the symmetry of the graph, or by extending the paths across columns using the $7$ paths from the proof of Theorem \ref{theorem:gonality_kings}.

\begin{figure}[hbt]
    \centering
    \includegraphics[scale=0.9]{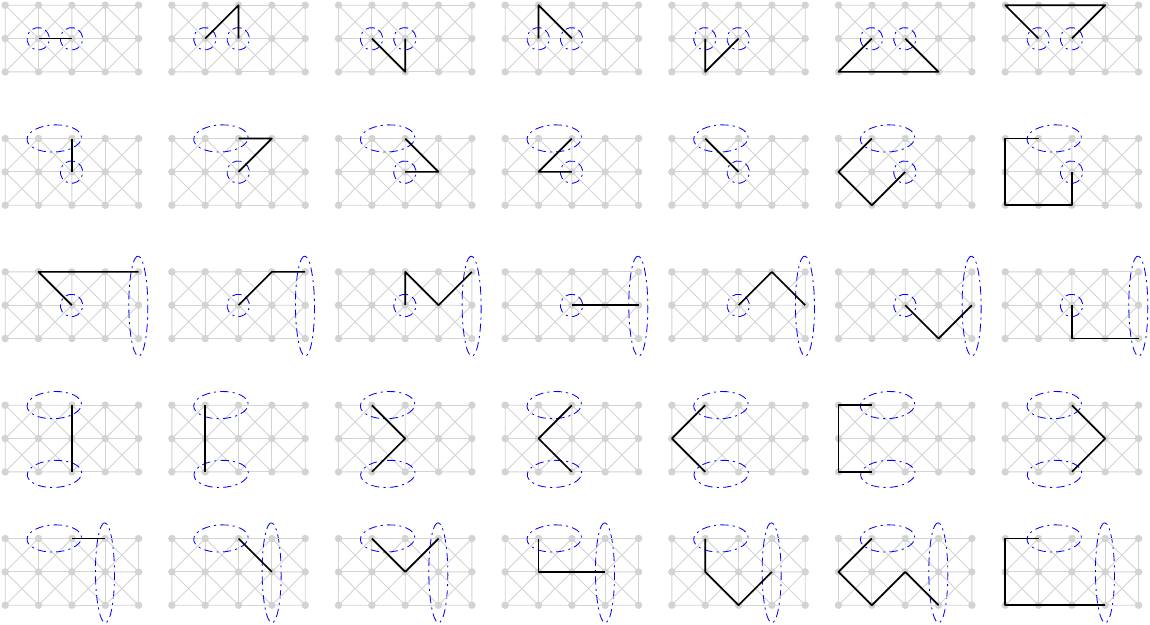}
    \caption{Five pairs of eggs from the scramble on $K_{3\times 5}$, and for each pair $7$ pairwise edge-disjoint paths connecting them.}
    \label{figure:sample_eggs_k35}
\end{figure}
\end{proof}

\begin{lemma}\label{lemma:appendix _k34}
    We have $\gon(K_{3\times4})=7$.
\end{lemma}

\begin{proof} We have $\gon(K_{3\times4})\leq 7$ by Theorem  \ref{appendix:kings}.  Suppose for the sake of contradiction that $\gon(K_{3\times4})\leq 6$.  Then there exists a positive rank divisor $D$ of degree $6$ on $K_{3\times4}$.  As presented in the discussion after \cite[Lemma 16]{aidun2019gonality}, we may assume that $D$ is effective, that $D$ places at most $\val(v)-1$ chips on each vertex, and that $D$ does not place $\val(v)-1$ chips and $\val(w)-1$ on any pair of adjacent vertices $v$ and $w$.

The remainder of our argument will consist of carefully choosing an unchipped vertex of $K_{3\times4}$ and running Dhar's burning algorithm from there.  We will argue that the whole graph burns, contradicting $D$ having positive rank.  We make the following observations for how fire will propagate throughout the graph.
\begin{itemize}
    \item[(1)] If four vertices forming a $K_4$ subgraph have together either at most $2$ chips, or three chips placed in a $2$-$1$ pattern, then once one vertex burns, the whole $K_4$ subgraph will burn.

    \item[(2)] There exists a vertex $v$ with no chips such that running Dhar's from $v$ will burn the whole column of $v$.

    Certainly if some column has no chips, this is true: pick $v$ in that column.  Otherwise, every column has at least one chip.  The distribution of chips among the four columns is thus either $1$, $1$, $1$, and $3$ or $1$, $1$, $2$, and $2$, in some order.  Either way we can find two adjacent columns, $C_i$ and $C_j$, such that $C_i$ has $1$ chip and $C_j$ has at most $2$ chips.  We may view $C_i\cup C_j$ as two copies of $K_4$, overlapping at the two vertices $v_{i,2}$ and $v_{j,2}$  If both copies of $K_4$ have $3$ chips, then $D=v_{i,2}+2v_{j,2}$; running Dhar's from any other vertex in $C_i\cup C_j$ burns all of $C_i\cup C_j$.  Thus we may assume one copy of $K_4$ has at most $2$ chips.  Choose $v$ to be an unchipped vertex on that $K_4$.  By observation (1), that copy of $K_4$ burns.  The remaining vertex in $C_i$ has at most $1$ chip and at least two burning edges, so it burns, giving us a burned column.

    \item[(3)] If the second column burns, then the first column will burn.  (Symmetrically if the second column burns, then the first column burns.)
    
    To see this, note that the first column has at most $6$ chips.  For all three vertices to remain unburned, $7$ chips are required.  For two vertices to remain unburned, either the two corner vertices need valence-many chips (a contradiction), or the middle and a corner need $4$ and $2$ chips, (also a contradiction, since no two adjacent vertices  have one less than valence-many chips).  And for one vertex to remain unburned, that vertex must have valence-many chips, again a contradiction.

    \item[(4)] If the first column burns, then the second column will burn.  (Symmetrically if the fourth column burns, then the third column burns.)

    There are at most $6$ chips in $C_2$, and the vertices have $2$, $3$, and $2$ incoming burning edges from $C_1$; thus at least one vertex in $C_2$ must immediately burn.  First we deal with the case that exactly one immediately burns.  If $v_{2,1}$ is the only one that immediately burns, then the divisor must be $4v_{2,2}+2v_{2,3}$.  The fire spreads to $C_3$, leading to the other two vertices burning.  A similar argument holds if $v_{2,3}$ immediately burns.  If $v_{2,2}$ is the only that immediately burns, then the divisor must be $3v_{2,1}+3v_{2,3}$.  The fire spreads to $C_3, $ leading to the other two vertices burning. In all cases, $C_2$ burns.

    Now we deal with the case that exactly two immediately burn.  If $v_{2,1}$ and $v_{2,2}$ immediately burn, then $v_{2,3}$ must have at least $3$ chips, and also at most $4$ chips.  It follows that $C_3$ contains at most $3$ chips.  Consider the copy of $K_4$ made up of $v_{2,1}, v_{2,2}, v_{3,1}$, and $v_{3,2}$.  If the whole $K_4$ burns, then for $v_{2,3}$ to stay unburned then so too must $v_{3,3}$.  This means that $D=4v_{2,3}+2v_{3,3}$.  However, the fire spreads to $C_4$ and from there burns $v_{3,3}$, and from there $v_{2,3}$. so all of $C_3$ burns.  If that $K_4$ does not burn, then either $v_{3,1}$ or $v_{3,2}$ has three chips.  This forces either $D=3v_{2,3}+3v_{3,1}$ or $D=3v_{2,3}+3v_{3,2}$; in both cases $v_{3,3}$ burns, and from there $v_{2,3}$ burns.  A similar argument holds if $v_{2,2}$ and $v_{2,3}$ immediately burn.  Finally, if $v_{2,1}$ and $v_{2,3}$ immediately burn, then $v_{2,2}$ must have at least $5$ chips.  This means $C_3$ has at most one chip, and since all its vertices neighbor either $v_{2,1}$ or $v_{2,3}$, at least two vertices in $C_3$ burn.  This means $7$ vertices neighboring $v_{2,2}$ burn, and since $\deg(D)=6$, we have that $v_{2,2}$ must burn, completing the column.

    \item[(5)] If the second column burns, then the third column will burn.  (Symmetrically if the third column burns, then the second column burns.)
    
    The argument is nearly identical that of claim (4), with the column of each vertex shifted right by one.  The one difference comes in when we find that $D=4v_{3,3}+2v_{4,3}$; here we reach a contradiction, since no two adjacent vertices can have valence-minus-one chips on them.
\end{itemize}
Taking together, we quickly have a proof:  choose a vertex $v$ satisfying (2).  Observations (3), (4), and (5) guarantee that once a column burns, so do all adjacent columns, meaning that running Dhar's burning algorithm starting at $v$ burns the entire graph.  Since $v$ has no chips, this is a contradiction to $D$ being a positive rank divisor.  We conclude that $\gon(K_{3\times 4})\geq 7$, and thus $\gon(K_{3\times 4})=7$.
\end{proof}

\section{Computations for knight's graphs}\label{appendix:knights}

\begin{proof}[Proof of Theorem \ref{theorem:knights_3_and_4_and_5}(i)] First assume that $n$ is odd. Since $n>3$, we know $N_{3\times n}$ is connected.  Consider the divisor $D$ that places two chips on every other column, starting with the second, on its top and bottom vertices.  This divisor has degree $2\lfloor n/2\rfloor$, which since $n$ is odd is equal to $2\lfloor (n-1)/2\rfloor$.  We will show that this divisor has positive rank.

\begin{figure}
    \centering
    \includegraphics{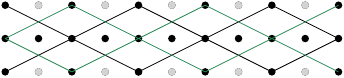}
    \caption{The divisor $D$ for $N_{3\times n}$ with $n$ odd, which places one chip on each gray vertex; the components of $G-\textrm{supp}(D)$ are illustrated.}
    \label{figure:3xn_odd_case}
\end{figure}

In Figure \ref{figure:3xn_odd_case} we illustrate the placement of chips from $D$ as gray vertices, with the connected components of $G-\textrm{supp}(D)$ illustrated as well.  Some of these connected components are isolated vertices $v$; chips can be moved to them by set-firing $V(G)-\{v\}$.  There are two other components, consisting of $4$-cycles glued at vertices.  Focus on one such component.  Firing the complement of the component moves $1$ chips to each of its vertices in the top and bottom rows.  At this point every vertex from that component in the middle row has a chip on each neighbor; firing the complement of that vertex moves a chip there.  A symmetric argument shows we may move a chip to any vertex in the other large component.  Taken together, a chip may be moved to any vertex, so our divisor has positive rank as desired.

Now consider $n$ even.  The case of $n=4$ is handled by Proposition \ref{prop:knights_exact_3_by_small} (proved later in this appendix), so we consider $n\geq 6$.  We use a similar chip placement to the odd case, in that we choose a set of columns and place two chips on each, one on the top vertex and one on the bottom vertex.  However, the columns we choose are the $3^{rd}$, the $4^{th}$, and the $k^{th}$ for any odd $k\geq 7$.  This divisor $D$ has degree $2\lfloor (n-1)/2\rfloor$ since $n\geq 6$.

\begin{figure}
    \centering
    \includegraphics{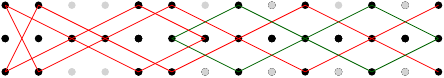}\quad\quad\includegraphics{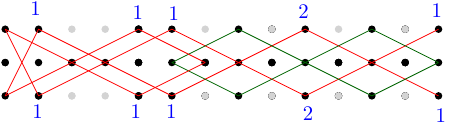}
    \caption{On the left, the divisor $D$ for $N_{3\times n}$ with $n$ even, which places one chip on each gray vertex; the components of $G-\textrm{supp}(D)$ are illustrated. On the right, the equivalent divisor obtained by firing the complement of the larger component of $G-\textrm{supp}(D)$.}
    \label{figure:3xn_even_case_knights}
\end{figure}

We represent on the left in Figure \ref{figure:3xn_even_case_knights} the divisor as grayed vertices, with the connected components of the remaining vertices illustrated.  Moving chips to most vertices works largely as in the odd case: isolated vertex components are easily handled, as is the smaller of the two larger components.  For the last component, firing its complement places chips on all its vertices except those in the middle row, and except for the top and bottom vertex in the rightmost row; this is illustrated in the right on Figure \ref{figure:3xn_even_case_knights}.  Every vertex not yet chipped either has all neighbors with a chip, and so can be dealt with by firing its complement; or it is one of $(1,1)$, $(3,1)$, or $(3,2)$, and firing the complement of this set places chips on all those vertices.
\end{proof}

\begin{proof}[Proof of Theorem \ref{theorem:knights_3_and_4_and_5}(ii)]  Consider the following divisor $D$ on the $4\times n$ knight's graph.  Choose one vertex from each column, alternating between the second and third rows.  Place $1$ chip on that vertex if it is in the two leftmost or two rightmost columns; otherwise place $2$ chips on it. Note that $\deg(D)=2n-4$.  We will argue that $D$ has positive rank.

\begin{figure}[hbt]
    \centering
\includegraphics{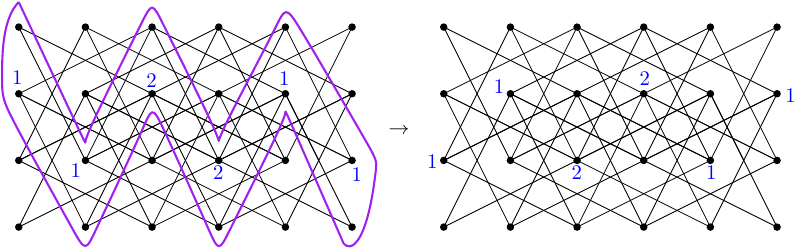}
    \caption{A divisor $D$ on $N_{4\times n}$ with a set-firing move to yield an equivalent divisor $D'$.}
    \label{figure:4xn_positive_rank}
\end{figure}

Illustrated in Figure \ref{figure:4xn_positive_rank} is the divisor $D$, as well as a set-firing move transforming it into another divisor $D'$.  Any vertex in the second or third row has a chip placed on it by $D$ or $D'$.  For every vertex $v$ in the first or fourth row, either $D$ or $D'$ places a chip on all of its neighbors.  Starting from that divisor and set-firing $V(G)-\{v\}$ places a chip on $v$.  Having argued that we may place a chip on any vertex, we conclude that $D$ has positive rank.
\end{proof}

\begin{proof}[Proof of Theorem \ref{theorem:knights_3_and_4_and_5}(iii)]  Consider the following divisor $D$ on the $4\times n$ knight's graph.  Place a chip on every vertex in the third row.  Then place one more chip in each column, either in the second or fourth row, in groups of four (so the first four columns have a chip in the second row, the next four columns have a chip in the fourth row, and so on).  Finally, on every fifth row, place one more chip in either the second or fourth row, whichever does not already have a chip. This divisor has degree $2n+\lfloor n/5\rfloor$.  We will argue it has positive rank.

\begin{figure}[hbt]
 \centering\includegraphics{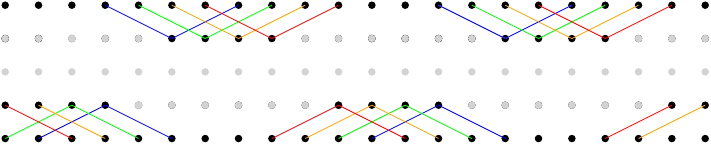}
    \caption{A divisor on $N_{5\times n}$ (one chip on each gray vertex), and the connected components of $G$ with the chipped vertices deleted.}
    \label{figure:5xn_positive_rank}
\end{figure}

We represent the divisor $D$ in Figure \ref{figure:5xn_positive_rank} with gray vertices, and illustrated the connected components of $G-\textrm{supp}(D)$.  Note that these connected components are all trees, and that each vertex in $\textrm{supp}{D}$ has at most one neighbor in each tree.  It follows from the argument in \cite[Theorem 3.33]{db} that $D$ has positive rank (indeed, firing the complement of one of the trees moves chips onto all vertices of that tree).
\end{proof}

\begin{proof}[Proof of Proposition \ref{prop:knights_exact_3_by_small}]
    The graph $N_{3\times 3}$ is disconnected, with a cycle on eight vertices and an isolated vertex as its two components.  By Lemma \ref{lemma:common_graph_gonalities}, these components have gonality $2$ and gonality $1$, respectively, so $\gon(N_{3\times 3})=2+1=3$.

The graph $N_{3\times 4}$ is illustrated with two different pictures in Figure \ref{figure:knights_3by4}; we focus on the second rendering, which appears as a $3\times 4$ grid graph with edges missing from the middle row.   A divisor $D$ of degree $2$ is illustrated on the right in Figure \ref{figure:knights_3by4}. This divisor has positive rank: by firing the first column in this grid interpretation, then the first two columns, and so on, we may move two chips to the top and bottom of any column.  To place a chip on a vertex $v$ in the middle row, first move chips to its two neighbors, then set-fire the complement of $v$.  This gives $\gon(N_{3\times 4})\leq \deg(D)=2$.  Since $N_{3\times 4}$ is not a tree, we know $\gon(N_{3\times 4})\geq 2$ by Lemma \ref{lemma:common_graph_gonalities}, so  $\gon(N_{3\times 4})=2$.

\begin{figure}
    \centering
\includegraphics{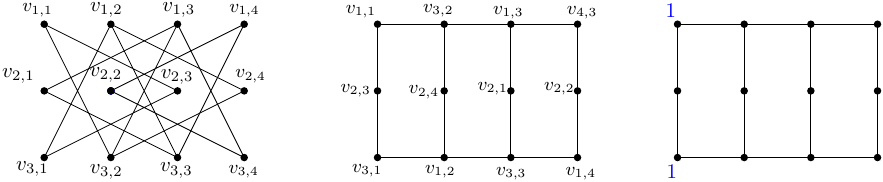}
    \caption{Two depictions of the $3\times 4$ knight's graph, with a positive rank divisor.}
    \label{figure:knights_3by4}
\end{figure}

For $N_{3\times 5}$ and $N_{3\times 6}$, an upper bound of $4$ on gonality comes from Theorem \ref{theorem:knights_3_and_4_and_5}.  For a lower bound we will use scramble number.

The graph $N_{3\times 5}$ is illustrated in its usual form on the left in Figure \ref{figure:knights_3by5}, with an alternate drawing in the middle.  Since scramble number is invariant under ``smoothing over'' $2$-valent vertices \cite[Proposition 4.6]{new_lower_bound}, the scramble number of $N_{3\times 5}$ is equal to the scramble number of the graph $G$ illustrated on the right in Figure \ref{figure:knights_3by5}.  That same figure presents a scramble $\mathcal{S}$ on $G$, which we claim has order $4$.  Certainly $h(\mathcal{S})=4$, since there are four disjoint eggs.  Between any pair of eggs we can find four pairwise edge-disjoint paths, so $e(\mathcal{S})\geq 4$.  Thus $||\mathcal{S}||\geq 4$.  This gives us
\[4=||\mathcal{S}||\leq \sn(G)=\sn(N_{3\times 5})\leq \gon(3\times 5)\leq 4,\]
allowing us to conclude that $\sn(N_{3\times 5})=\gon(N_{3\times 5})=4$.

\begin{figure}[hbt]
    \centering
\includegraphics{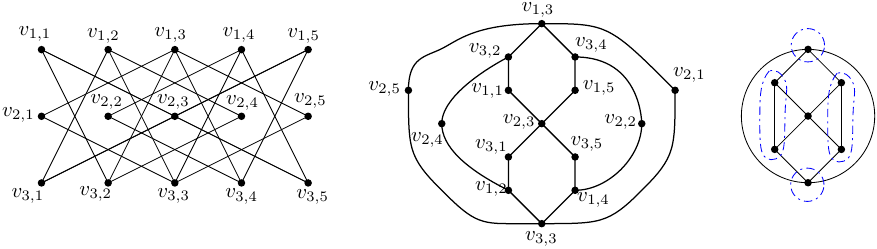}
    \caption{Two drawings of the $3\times 5$ knight's graph, and the graph obtained by smoothing over $2$-valent vertices, with a scramble.}
    \label{figure:knights_3by5}
\end{figure}

For the $3\times 6$ knight's graph, we note that $N_{3\times 5}$ is a subgraph of $N_{3\times 6}$.  Since scramble number is monotone under taking subgraphs \cite[Proposition 4.5]{new_lower_bound}, we have
\[4=\sn(N_{3\times 5})\leq \sn(N_{3\times 6})\leq \gon(N_{3\times 6})\leq 4,\]
so $\sn(N_{3\times 6})=\gon(N_{3\times 6})=4$.
\end{proof}

\end{document}